\documentclass[11pt,a4paper,reqno]{amsart}

\usepackage[all]{xy}

\usepackage{graphicx}

\usepackage{tikz}
\usetikzlibrary{arrows,snakes}

\usepackage{undertilde}

\usepackage{latexsym,amssymb,amsmath}

\usepackage{enumerate}
\usepackage{amsmath}

\usepackage{amsfonts}
\usepackage{amssymb}
\usepackage{eucal}
\usepackage{url}
\usepackage{amssymb}

\newcommand{\ria}{\rightarrow}

\newcommand{\fao}[1]{\forall #1 \, }

\newcommand{\ape}{\hat{\ }}

\newcommand{\CCC}{\mathcal{C}}

 \newcommand{\vsps}{\vspace{3pt}}

\newcommand{\ttext}[1]{\ \text{#1}\ }

\newcommand{\SI}[1]{\Sigma^0_{#1}}
\newcommand{\PI}[1]{\Pi^0_{#1}}

\newcommand{\PPI}{\PI{1}}

\newcommand{\Q}{\mathcal Q}

\newcommand{\Kuc}{Ku{\v c}era}
\newcommand{\ML}{Martin-L{\"o}f}

\newcommand{\strcantor}{2^{<\omega}}
\newcommand{\seqcantor}{2^{ \omega}}
\newcommand{\cantor}{\seqcantor}

\newcommand{\NN}{{\mathbb{N}}}  
\newcommand{\RR}{{\mathbb{R}}}
\newcommand{\QQ}{{\mathbb{Q}}}

\newcommand{\ZZ}{{\mathbb{Z}}}

\newcommand{\DII}{\Delta^0_2}

\newcommand{\ex}{\exists}
\newcommand{\fa}{\forall}
\newcommand{\LR}{\Leftrightarrow}

\newcommand{\LLR}{\ \Leftrightarrow \ }
\newcommand{\RA}{\Rightarrow}
\newcommand{\RRA}{\ \Rightarrow\ }
\newcommand{\LA}{\Leftarrow}

\newcommand{\UA}{\uparrow}
\newcommand{\n}{\noindent}
\newcommand{\wt}{\widetilde}
 
\newcommand{\sub}{\subseteq}

\newcommand{\ol}{\overline}

\newcommand{\ul}{\underline}

\renewcommand{\land}{\&}
\renewcommand{\lor}{\vee}
\newcommand{\ES}{\emptyset}

\renewcommand{\hat}{\widehat}

\newcommand{\lland}{\ \land \ }
\newcommand{\llor}{\ \lor \ }

\newcommand{\itone}{\item[(i)]}
\newcommand{\ittwo}{\item[(ii)]}
\newcommand{\itthree}{\item[(iii)]}
\newcommand{\itfour}{\item[(iv)]}

\newcommand{\la}{\langle}
\newcommand{\ra}{\rangle}

\newcommand{\bi}{\begin{itemize}}
\newcommand{\ei}{\end{itemize}}
\newcommand{\bc}{\begin{center}}
\newcommand{\ec}{\end{center}}

\newcommand{\sss}{\sigma}
\newcommand{\aaa}{\alpha}

\newcommand{\sssl}{|\sigma|}

\newcommand{\leb}{\mathbf{\lambda}}
\newcommand{\eps}{\epsilon}
\newcommand{\Om}{\Omega}
\newcommand{\twoset}{\{0,1\}}

\newcommand{\sN}[1]{_{#1\in \NN}}

\newcommand{\estring}{\emptyset}

\newcommand{\Opcl}[1]{[#1]^\prec}
\newcommand{\tp}[1]{2^{#1}}

\newcommand\+[1]{\mathcal{#1}}

\newcommand{\uhr}[1]{\!  \upharpoonright_{#1}}

\newtheorem{theorem}{Theorem}[section]
\newtheorem{thm}[theorem]{Theorem}

\newtheorem{definability lemma}[theorem]{Definability Lemma}

\newtheorem{fact}[theorem]{Fact}

\newtheorem{prop}[theorem]{Proposition}
\newtheorem{claim}[theorem]{Claim}

\newtheorem{lemma}[theorem]{Lemma}

\newtheorem{cor}[theorem]{Corollary}

\theoremstyle{definition}

\newtheorem{deff}[theorem]{Definition}

\newtheorem{remark}[theorem]{Remark}

\newcommand{\vsp}{\vspace{6pt}}

\DeclareMathSymbol{\wtilde}{\mathord}{largesymbols}{"65}

\newcommand{\Mart}{\mbox{\rm \textsf{mart}}}
\newcommand{\cdf}{\mbox{\rm \textsf{cdf}}}
\usepackage{ifthen}
\newcommand{\Com}[1]{ \ifthenelse{\boolean {comments}}{ \vsp \n  {\bf  Comment:}   {   #1}}{}   }
\newboolean{comments}	

\setboolean{comments}{false}

\newcommand{\rem}[1]{\relax}
\usepackage{ifthen, xspace}
    \newcommand{\bsc}{\usefont{T1}{cmr}{bx}{sc}}
    \newcommand{\pending}[1][]{{\noindent\highlight{\bsc\ifthenelse{\equal{#1}{}}{[to be written]}{[to be written: {\rm #1}]}}}\xspace}
    \newcommand{\prelims}[1][]{{\noindent\highlight{\bsc\ifthenelse{\equal{#1}{}}{[add to prelims]}{[add to prelims: {\rm #1}]}}}\xspace}
\usepackage{color}
    \newcommand{\highlight}[2][red]{{\color{#1}#2}}

\newcommand{\R}{\mathbb{R}}
\renewcommand{\Q}{\mathbb{Q}}

\begin{document}

\title{Randomness and differentiability (long version)}

\date{\today---Started: Oct 15, 2009}

\author{Vasco Brattka}
\address{Vasco Brattka\\
Faculty of Computer Science, Universit\"at der Bundeswehr M\"unchen\\
85577 Neubiberg, Germany and\\
Department of Mathematics \& Applied Mathematics\\
University of Cape Town\\
Rondebosch 7701, South Africa}
\email{Vasco.Brattka@cca-net.de}

\author{Joseph S.~Miller}
\address{Joseph S.~Miller\\
Department of Mathematics\\
University of Wisconsin\\
Madison, WI 53706-1388, USA}
\email{jmiller@math.wisc.edu}

\author{Andr\'e Nies}
\address{Andr\'e Nies\\
Department of Computer Science, University of Auckland, Private bag 92019, Auckland, New Zealand}
\email{andre@cs.auckland.ac.nz}

\thanks{Brattka  was supported by the National Research Foundation of South Africa. Miller was supported by the National Science Foundation under grants DMS-0945187 and DMS-0946325, the latter being part of a Focused Research Group in Algorithmic Randomness. Nies was partially supported by the Marsden Fund of New Zealand, grant no.\ 08-UOA-187.}


\begin{abstract}
	We characterize some  major algorithmic randomness notions via differentiability of effective functions. 
	
	\n (1) As the main result we show that
a real number   $z\in [0,1]$ is computably  random if and only if  
  each  nondecreasing computable function $[0,1]\ria \R$  is differentiable at $z$. 
  
  \n (2) We prove that a~real number   $z\in [0,1]$ is weakly 2-random if and only if  
  each  almost everywhere differentiable  computable function $[0,1]\ria \R$  is differentiable at $z$. 
  
  \n  (3) Recasting in classical language results dating  from 1975  of  the constructivist Demuth, we     show that a real~$z$ is \ML\ random  if and only if     every  computable function of bounded variation is differentiable at~$z$, and similarly for absolutely continuous functions. 

 We also  use our analytic methods to show that  computable randomness of a  real is base invariant, and to derive other preservation results for randomness notions.
\end{abstract}

 \maketitle
 
 \tableofcontents

\section{Introduction}
The main thesis of this paper is that algorithmic randomness of a real is equivalent to differentiability of effective functions at the real. In more detail, for  every major algorithmic randomness notion, one can  provide a class of effective functions on the unit interval   so that 

\vsp
\n ($*$)  \  a real $z\in [0,1]$ satisfies the randomness notion  $\LLR$ 

\hfill each function in the class is differentiable at $z$.

\vsp

\n For instance, $z$ is computably  random $\LLR$ each computable nondecreasing function is differentiable at $z$. Furthermore,    $z$ is Martin-L\"of random $\LLR$ each computable function of bounded variation is differentiable at $z$. The second  result  was proved by Demuth \cite{Demuth:75}, who used constructive language; we will reprove it here in the usual language, using   the first result relativized  to an oracle set.


Classically, to say that a property holds for a ``random'' real $z \in [0,1]$ simply means that the reals failing the property form a null set. For instance, a well-known theorem of Lebesgue~\cite{Lebesgue:1909}  states that 
  every  nondecreasing function $f\colon \, [0,1] \to \RR$ is differentiable at all reals~$z$ outside a null set (depending on $f$).
 That is, $f'(z)$ exists for a random real $z$ in the sense specified above.  Via Jordan's result that each function of bounded variation is the difference of two nondecreasing functions (see, for instance, \cite[Cor 5.2.3]{Bogachev.vol1:07}), Lebesgue's theorem can be extended to functions of bounded variation.

In most  of the  results of the type ($*$)   above,  the implication ``$\RRA$''  can be seen as an effective form of Lebesgue's theorem. Before we make this precise, we will  provide some background on algorithmic randomness,  and    computable functions on the unit interval.
 
\subsection{Some background} \

\label{ss:prelimsAlgRdCompFcns}
\smallskip

\n \emph{Algorithmic randomness.} The idea in algorithmic randomness is to think of  a real as random if it is in no \emph{effective} null set. To specify an algorithmic  randomness notion, one has to specify   a type of effective null set, which is usually done by introducing a test concept. Failing the test is the same as being in  the null set.

 A  hierarchy of algorithmic randomness  notions has been developed, each one corresponding to certain aspects of our intuition.   Traditionally, the  central notion has been  \emph{\ML\ randomness}. A $\SI 1$ set $\+G \sub [0,1]$ has the form $\bigcup_m A_m$, where  $A_m$ is an open interval with dyadic rational endpoints obtained effectively from $m$. Let $\leb$ denote the usual Lebesgue measure on the unit interval. A~\emph{Martin-L\"of  test} is a   sequence of uniformly $\SI 1$ sets $(\mathcal G_m) \sN m$  in the unit interval such that $\leb \mathcal G_m \le \tp{-m}$ for each $m$. 
 The algorithmic null set it describes is $\bigcap_m \+ G_m$. 

Schnorr  \cite{Schnorr:75}  maintained  that   Martin-L\"of randomness    is  already   too powerful  to be considered algorithmic, because it is  based on computably enumerable objects as tests.   He  proposed a   weaker notion: a real   is called \emph{computably  random} if no computable betting strategy can win on its binary expansion (see Subsection~\ref{ss:randomness_notions} for detail).  We will see that   this is the appropriate notion for studying almost-everywhere differentiability  of important classes of computable functions.

  A $\PI 2$  set (or effective $G_\delta$ set) is of the form $\bigcap_m \+ G_m$,  where $(\mathcal G_m) \sN m$  is a   sequence of uniformly $\SI 1$ sets.  We call a   real \emph{weakly $2$-random} if it  is in no null $\PI 2$ set. Compared to \ML{} randomness, the test notion is relaxed by replacing   the condition  $\fa m \, \leb \mathcal G_m \le \tp{-m}$  above    by the weaker condition $\lim_m \leb \+G_m =0$.  For     background on algorithmic randomness see~\cite[Chapter~3]{Nies:book} or 
 \cite{Nies:ICM}. 

\smallskip

\n \emph{Computable functions on the unit interval.}
  Several  definitions of computability for a  function $f\colon [0,1] \ria \R$  have been 
    proposed.  In close analogy to the Church-Turing thesis, many (if not all)  of them  turned out to be equivalent. The common  notion comes close to being a generally accepted formalization of computability for functions on the unit interval.  Functions that     are intuitively computable, such as $e^x$ and $\sqrt x$, are  computable in this formal sense.  
Computable functions in that sense are necessarily continuous; see the discussion in Weihrauch \cite{Weihrauch:00}.
The generally accepted   notion goes back to work of   Grzegorczyk and Lacombe from the 1950s, as discussed in   Pour-El and Richards prior to   \cite[Def.\  A, p.\ 25]{Pour-El.Richards:89}.  In  the same  book   \cite[Def.\  C, p.\ 26]{Pour-El.Richards:89}   they give   a simple    condition  equivalent to    computability of $f$, which they      call    ``effective Weierstrass'':

  \vsps

\n $f\colon [0,1] \ria \R$ is computable $\LR$ 
 there is an effective sequence  $(P_n) \sN n$  of polynomials with rational coefficients  such that $||f - P_n ||_\infty  \le \tp{-n}$ for each~$n$.

  \vsps
\n  This can be interpreted as saying that $f$ is a computable point in a suitable computable metric space. 
See Subsection~\ref{ss:compfunctions} for another characterization.

\subsection{Results of   type ($*$): the  implication ``$\RA$''} \

\n {(a)} 	 We will show in Theorem~\ref{thm:CRdiff} that

	 \vsp

	\n a real $z\in [0,1]$ is computably random $\RA$ 

	\hfill each nondecreasing computable function   is differentiable at $z$.

	\vsps This is  an effectivization of Lebesgue's theorem in terms of the concepts given above. Lebesgue's theorem is usually  proved via  Vitali coverings. This method is non-constructive; a new approach is needed for the effective version.
The proof is  by contraposition. The main problem is to proceed from the non-existence of $f'(z)$, which is based on the behaviour of slopes at arbitrarily small intervals $I$ containing $z$, to the success of a betting strategy, which only has access to basic dyadic intervals (namely, intervals of the form $[i \tp{-n}, (i+1)\tp{-n})$ for $ n \in \NN, i < \tp n$). The solution is to bet with    scaled and shifted  basic dyadic intervals, and show that the scaling and shifting parameters taken  from a finite set are sufficient to approximate $I$   from the outside and  also from the inside by such intervals. 

\n {(b)} The corresponding result of Demuth~\cite{Demuth:75}  involving  \ML{} randomness and computable functions of bounded variation will be  re-obtained as a corollary, using an effective form of Jordan's theorem.   We note that Demuth's    proof is somewhat obscure, which is partly due to the fact that it is uses   constructive language and notation. The attribution to Demuth  relies on  an interpretation, rather than a straightforward reading, of~\cite{Demuth:75}.

\n {(c)}  For  weak $2$-randomness, we take the largest class of computable functions that makes sense in this setting: the almost everywhere differentiable computable functions. The implication $\RA$ is obtained by observing that the points of nondifferentiablity for any computable function is a $\SI 3$ set (i.e., an effective   $G_{\delta \sigma}$ set). If the function is a.e.\ differentiable, this set is null, and hence  cannot contain a weakly $2$-random real.  

\subsection{Results of   type ($*$): the    implication ``$\LA$''} \

\n This is typically proved by contraposition.  One  simulates tests by non-differentiability of functions. Thus,  given a test in the sense of the algorithmic randomness notion, one builds a computable function $f$ on the unit interval such that, for each real $z$ failing the test,  $f'(z)$ fails to  exist.  We will provide direct, uniform constructions of this kind for weak $2$-randomness (c),  and then for \ML{} randomness (b). The computable functions we build are   sums of ``sawtooth functions''. For computable randomness (a), the simulation is less direct, though still uniform. The results in more detail are as follows.

\vsps

\n (a)  For  each  real $z$ that is not   computably random, there is a computable nondecreasing   function  $f$   such that  $\ol D f(z) = \infty$ (Theorem~\ref{thm:CRdiff}).

\n (b)    There is, in fact,  a single  computable function $f$  of  bounded variation such that $f'(z)$ fails to exist for   all non-\ML{} random reals~$z$  (Lemma~\ref{lem:MLtest_to_function}).

\n (c)  For each   $\PI 2$ null set there is an a.e.\ differentiable computable function $f$ that is non-differentiable at any $z$ in the null set (Theorem~\ref{thm:w2rChar}).

 As mentioned above,  (b)  was  already stated by  Demuth~\cite[Example 2]{Demuth:75}. For background on Demuth's work see the survey~\cite{Kucera.Nies:12}.

The  implication ``$\LA$''  is also rooted in  results from classical analysis. For instance, Zahorski~\cite{Zahorski:46} proved that each null $G_\delta$ subset of $\RR$  is the non-differentiability set of a  monotonic Lipschitz function. For a recent proof, see Fowler and Preiss~\cite{Fowler.Preiss:09}.

\subsection{Classes of effective functions, and randomness notions} \

\n
The results of   type ($*$) mean  that    all the major algorithmic randomness notions for a real can now  be  matched with at least one class of  effective functions  on the unit interval  in such a way that randomness of a real  is equivalent to differentiability at the real.
The  analytical properties of  functions we use  are the well-known ones from classical real analysis. 

The matching is onto,  but not 1-1: in a sense, randomness notions are coarser than classes of effective functions. Computable randomness is characterized  not only by differentiability of  nondecreasing computable functions, but also of computable Lipschitz  functions  \cite{Freer.Kjos.ea:nd}. Furthermore,   as an    effectiveness condition on functions,  one can   choose  anything between  computability in the sense discussed in Subsection~\ref{ss:prelimsAlgRdCompFcns} above, and  the weaker condition  that  $f(q)$  is a computable real (see Subsection~\ref{ss:CompReals}), uniformly in a rational $q \in [0,1]$. Several notions lying in  between have received attention. One of them is Markov computability, which will be discussed briefly in Section~\ref{s:extensions}. Note that for nondecreasing continuous functions, the effectivity notions coincide by Proposition~\ref{prop:monotonic computable}.

A further well-studied algorithmic randomness notion is  Schnorr randomness, which is even weaker than computable randomness (see, for instance, \cite[Section 3.5]{Nies:book}). A \emph{Schnorr test}  is  a Martin-L\"of test $(\mathcal G_m) \sN m$ such that $\leb \+G_m$ is a computable real uniformly in $m$. A real $z$ is Schnorr random if $z \not \in \bigcap_m \+G_m$ for each Schnorr test  $(\mathcal G_m) \sN m$.

To characterize Schnorr randomness in terms of differentiability, we need a stronger notion of effectivity for functions. Call a function $f$ \emph{variation computable} if it is a computable point in the Banach space $AC_0[0,1]$ of absolutely continuous functions vanishing at $0$, where the norm of a function is its variation on $[0,1]$. The computable structure (in the sense of \cite[Ch.\ 2]{Pour-El.Richards:89}) is given, for instance,  by the polynomials with rational coefficients. Thus,  $f\in AC_0[0,1]$ is variation computable iff for each $n$,  one can determine   a polynomial $P_n$ with rational coefficients,  vanishing at $0$,   such that the variation of $f-P_n$ is at most $\tp{-n}$.  By the effective version of a classical theorem from analysis (see, for instance, \cite[Ch.\ 20]{Carothers:00}), $AC_0[0,1]$ is effectively isometric with the space $(\mathcal L_1[0,1], | |. | |_1)$, where  the computable structure is also determined by the  polynomials with rational coefficients. The isometry is given by differentiation, and its inverse by the indefinite integral.

Recent results of J.\ Rute \cite{Rute:12}, and independently Pathak, Rojas and Simpson, can be restated as follows:  $z$ is Schnorr random $\LR$  each  absolutely continuous function that is computable in the  variation norm  is differentiable at~$z$. Freer, Kjos-Hanssen, and Nies \cite{Freer.Kjos.ea:nd} showed the analogous result for  Lipschitz functions.


The matching between algorithmic randomness notions and classes of effective functions  is  summarized in Figure~\ref{fig:diagram}. 

\tikzstyle{header} = [font=\bfseries]
\tikzstyle{notion} = []
\tikzstyle{class} = [text centered]
\tikzstyle{comp} = [midway, right=0cm, text centered, text width=2.4cm, font=\small\itshape]

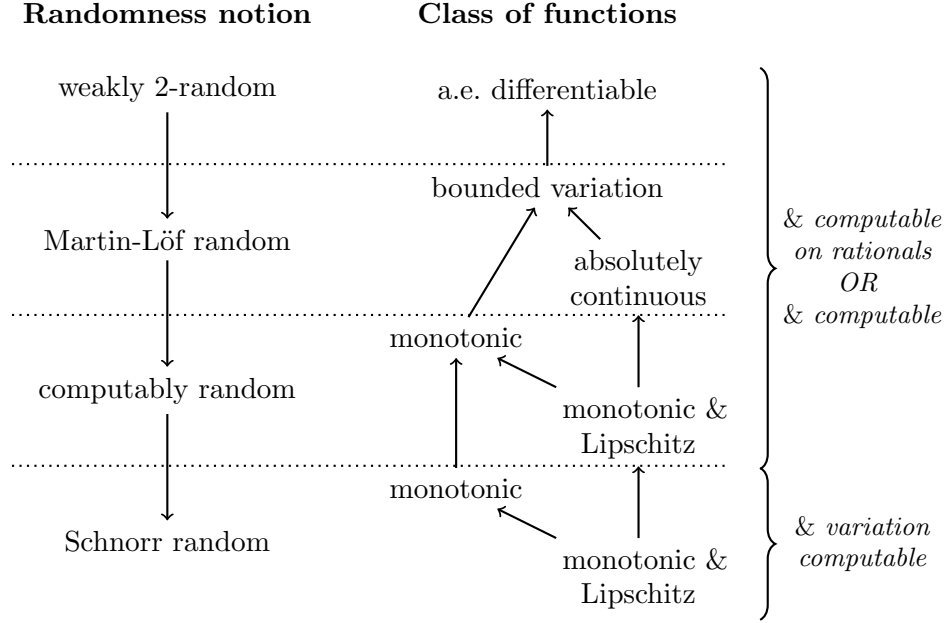
\begin{figure}[htbp] 
\begin{tikzpicture}[node distance=2cm, thick, segment amplitude=0.2cm]
	\node[header] (rn) {Randomness notion};
	\node[notion, below of=rn, yshift=1cm] (w2r) {weakly $2$-random};
	\node[notion, below of=w2r] (mlr) {\ML\ random};
	\node[notion, below of=mlr] (cr) {computably random};
	\node[notion, below of=cr] (sr) {Schnorr random};
	\path[->] (w2r) edge (mlr);
	\path[->] (mlr) edge (cr);
	\path[->] (cr) edge (sr);
	
	\node[header, right of=rn, xshift=3cm] (cf) {Class of functions};
	\node[class, below of=cf, yshift=1cm] (aed) {a.e.\ differentiable};
	\node[class, below of=aed, yshift=0.7cm] (bv) {bounded variation};
	\node[class, below of=bv, xshift=1.2cm, yshift=0.8cm, text width=2cm] (ac) {absolutely continuous};
	\node[class, below of=bv, xshift=-1.2cm] (m1) {monotonic};
	\node[class, below of=ac, text width=2cm] (mL1) {monotonic~\& Lipschitz};
	\node[class, below of=m1] (m2) {monotonic};
	\node[class, below of=mL1, text width=2cm] (mL2) {monotonic~\& Lipschitz};
	
	\path[->] (bv) edge (aed);
	\path[->] (ac) edge (bv);
	\path[->] (m1) edge (bv);
	\path[->] (mL1) edge (ac);
	\path[->] (mL1) edge (m1);
	\path[->] (m2) edge (m1);
	\path[shorten >=-1pt, ->] (mL2) edge (mL1);
	\path[->] (mL2) edge (m2);
	
	\draw[snake=brace, segment amplitude=0.2cm] ([xshift=2.8cm]aed.north) -- ([xshift=1.6cm]mL1.south) node[comp] {$\&$ computable on~rationals\break OR \break $\&$ computable};
	\draw[snake=brace] ([xshift=4cm]m2.north) -- ([xshift=1.6cm]mL2.south) node[comp] {$\&$ variation computable};
	
	\draw[dotted] ([yshift=-2cm]rn.west) -- ([yshift=-2cm, xshift=0.5cm]cf.east);
	\draw[dotted] ([yshift=-4cm]rn.west) -- ([yshift=-4cm, xshift=0.5cm]cf.east);
	\draw[dotted] ([yshift=-6cm]rn.west) -- ([yshift=-6cm, xshift=0.5cm]cf.east);
\end{tikzpicture}
\caption{Randomness notions matched with classes of effective functions defined on $[0,1]$ so that ($*$) holds}
\label{fig:diagram}\end{figure}


\subsection{Discussion}

The results above indicate a rich two-way  interaction   between algorithmic randomness and analysis.

\n \emph{Analysis to randomness:}  Characterizations via differentiability  can be used to improve our understanding of  an  algorithmic randomness notion. For instance, we  will show that  several randomness notions of reals are   preserved  under the maps  $z \to  z^\aaa$ where $\aaa\neq 0 $  is a computable  real.   Furthermore, we show that  computable randomness of a real is base invariant:  it does not depend on the fact that one uses the  binary expansion  of a real in its  definition  (Theorem~\ref{thm:CRbaseinv} below).


\n \emph{Randomness to analysis:}
the results also improve our  understanding of the underlying classical theorems. They indicate that in the setting of Lebesgue's theorem mentioned close to  the beginning of the paper, 
the exception sets for differentiability of nondecreasing functions are simpler than the exception sets for  functions of  bounded variation. Furthermore,   one  can  attempt to  calibrate, in the sense of reverse mathematics,  the strength of theorems saying that a certain function is  a.e.\ well behaved.    The benchmark principles have the form ``for each oracle set  $X$,  there is a set $R$ that is  random in $X$'', for some fixed algorithmic randomness notion.  For \ML{} randomness, the  principle  above is called ``weak weak K\"onig's Lemma''.   Now consider the principle that 
every function of bounded variation is differentiable at some real. By the work in Subsection~\ref{ss:Proof_MLbV} below, this implies weak weak K\"onig's Lemma over a standard  base theory called RCA$_0$. Recent work of Nies and Yokoyama  (see \cite[Part 2]{LogicBlog:13}) uses genuine methods of  reverse mathematics  to show that  the converse implication holds as well. One can also study the strength of the Lebesgue differentiation theorem; see \cite[Section 3.4]{Montalban:11} for some background.

\subsection{Structure of the paper}
 Section~\ref{s:background} provides background from computable analysis.  Section~\ref{s:CRand_intro} introduces computable randomness and shows its base invariance. The central Section~\ref{s:CRandDiff}
     characterizes  computable randomness in terms of differentiability of computable functions.    The short Section~\ref{ss:Pathak} discusses some consequences of this  result. Section~\ref{s:Weak2_and_ML} characterizes weak $2$-randomness in terms of differentiability of computable functions, and provides the implication $\LA$  in ($*$) for \ML{} randomness.    The final Section~\ref{s:extensions}  extends  the results   to  functions that are merely  computable on the rationals, and to notions in between  such as  Markov computability which is introduced  briefly.  It also provides the implication $\RA$  in ($*$) for \ML{} randomness. The paper  ends with some open questions and future directions.

\section{Preliminaries on   computable analysis}
       \label{s:background}

 \subsection{Computable reals}  
 \label{ss:CompReals}

 A  sequence $(q_n)\sN n$ of rationals is called a \emph{Cauchy name} if   $|q_{n} - q_{k} | \le \tp{-n}$ for each $k \ge n$.
 If $\lim_n q_n =x$ we say that $(q_k)\sN k$ is a   \emph{Cauchy name} for $x$. Thus, $q_n$ approximates~$x$  up to an error  $|x - q_n | $ of at most $ \tp{-n}$.    
  Each $x \in [0,1]$ has a Cauchy name  $(q_n)\sN n$     such that  $q_0 = 0$,  $q_1=1/2$, and each  $q_n$ is of the form $i\tp{-n}$ for an integer $i$. Thus, if $n> 0$ then   $q_{n} - q_{n-1} = a \tp{-n}$ for some $a \in \Sigma  = \{-1, 0,1\}$.  In this way a real $x$ corresponds to an element of $\Sigma^\omega$. 
  (This is a name for $x$ in the  \emph{signed-digit representation} of  reals; see \cite{Weihrauch:00}.) 
 A real $x$ is called  \emph{computable} if it has  a computable Cauchy  name.  
 (Note that this definition is equivalent to the one  previously given,   that the  binary expansion be  computable as a sequence of bits; the binary expansion can, however, not be obtained uniformly from a   Cauchy  name,    since one would need to know whether the real is a dyadic rational. The definition given here is more convenient to work with.)
 
 A sequence $(x_n) \sN n$ of reals is \emph{computable} if $x_n$ is computable uniformly in~$n$. That is, there is a computable double sequence $(q_{n,k})_{n,k \in \NN}$ of rationals such that each   $x_n$ is a computable real as witnessed by its  Cauchy name  $(q_{n,k})\sN k$.

 
\subsection{More on computable functions defined on the    unit interval}
 
      \label{ss:compfunctions}
     
  Let $f \colon \, [0,1] \ria \R$.    The formal definition of computability closest to our intuition is perhaps the following:  there is a Turing functional that,   with a  Cauchy name $(q_n)\sN n$  of $x$   as an oracle, returns a Cauchy name  of $f(x)$. (Thus, equivalent Cauchy names for an  argument $x$ yield equivalent Cauchy names for the value $f(x)$.)   This  condition means that given  $k\in \NN$, using enough of the sequence $(q_n)\sN n$  we can compute  a rational $p$ such that $|x-p| \le \tp{-k}$.

 It is often easier to work with the equivalent Definition A in   Pour-El and Richards \cite[p.\ 26]{Pour-El.Richards:89}.     
 \begin{deff} \label{CompDef} A function  $f\colon \, [0,1] \to \mathbb R$ is called  \emph{computable} if

  \bi \item[(a)]  for each     computable    sequence of reals   $(x_k)\sN k$, the sequence   $f(x_k)$ is computable, and

 \item[(b)] $f$ is \emph{effectively  uniformly continuous}: there is a computable $h\colon \NN \ria \NN$ such that $|x-y| < \tp{-h(n)}$ implies $|f(x) - f(y) | < \tp{-n}$ for each $n$.  \ei
 \end{deff} 
 
 If $f$ is effectively  uniformly continuous,  we can replace (a) by 
the following apparently weaker  condition.
   
\bi \item[(a$'$)]  for {\it some}    computable    sequence of reals   $(v_i)\sN i$  that is  dense in $[0,1]$    the sequence $f(v_i)\sN i$ is computable.   \ei
Typically,  the sequence $(v_i)\sN i $ in (a$'$)  is  an effective  listing of the   rationals in $[0,1]$ without repetitions. 
To show   (a$'$) $\land$ (b) $\RA$ (a),   suppose  $(x_n )\sN n $ is a computable    sequence of reals  as witnessed by the computable double sequence $(q_{n,k})_{n,k \in \NN}$ of rationals (see Subsection~\ref{ss:CompReals}). Let $h$ be as in (b). We may assume that $h(p)\ge p$ for all $p$. To define the $(p-1)$-th approximation to $f(x_n)$,  find $i$ such that $|q_{n, h(p+1)} -v_i| < \tp{-h(p+1)}$, and  output the $p$-th approximation  $r$ to  
 $f(v_i)$. Observe  that 
 \[|f(x_n) - r| \le \tp{-p} + |f(x_n) - f(q_{n, h(p+1)}) | +  |f(q_{n, h(p+1)} ) - f(v_i)|  \le   \tp{-p+1}, \]
  as required. 

An \emph{index} for a computable function on the unit interval  $f$  is a pair consisting of a computable  index for the   double sequence  $(q_{n,k})_{n,k \in \NN}$ of rationals determining the values of $f$ at the  rationals, together with a computable  index for $h$.


\subsection{Computability for nondecreasing functions} 

\label{ss:comp_nondecreasing_names}

We will frequently work with nondecreasing functions. Mere continuity and (a$'$)  are   sufficient for such a function to be computable.    This easy fact will be very useful later~on.

   \begin{prop} \label{prop:monotonic computable}  Let $g$ be a nondecreasing function. Suppose there is  a computable dense sequence  $(v_i) \sN i$ of reals in $   [0,1]  $ such that the sequence of reals $ g(v_i)\sN i$ is  computable. Suppose that $g$ is also continuous.  Then     $g$ is computable. \end{prop}

\begin{proof} We analyze the usual  proof    that~$g$ is uniformly continuous in order to verify that $g$ is effectively uniformly continuous, as defined in   (b) of Definition~\ref{CompDef}. To define a function $h$ as in (b), given $n$ let $\eps = \tp{-n-2}$. Since $g$ is nondecreasing and continuous, we can compute  a      collection $\delta_i, v_i$  ($i \in F$) where $F\sub \NN$ is  finite  and the $\delta_i,v_i$ are    rationals in $[0,1]$,    such that  $[0,1] \sub \bigcup_{i \in F} B_{\delta_i}(v_i)$, and  $d(x,v_i) < \delta_i \ria  d(g(x), g(v_i)) < \eps$.  Let $\delta $ be the minimum distance of any pair of   balls $B_{\delta_i}(v_i)$ with a disjoint closure.  If  $d(x,y) < \delta$ then choose $i,k \in F$ such that $x \in B_{\delta_i}(v_i)$ and $y \in  B_{\delta_k}(v_k)$.  These two balls are not disjoint,  so we have $d(g(x), g(y)) <  4\eps = \tp{-n}$. Since we obtained   $\delta $ effectively from $n$, we can determine $h(n) $ such that $\tp{-h(n)} < \delta$.  \end{proof}

\subsection{Arithmetical complexity of sets of reals} 
\label{ss:arithmetical_complexity}
By an \emph{open interval} in $[0,1]$ 	we mean an interval of the form  $(a,b)$, $[0,b)$, $(a, 1]$ or $[0,1]$, where $0\le a\le b \le 1$. 
A~\emph{$\SI 1$ set} in $[0,1]$ is a set of the form $\bigcup_k A_k$ where  $(A_k)\sN k$ is a computable sequence of open intervals with dyadic rational endpoints.
	 A~\emph{$\PI 2$ set} has the form $\bigcap_m \+G$ where the $\+G_m$ are $\SI 1$ sets uniformly in $m$.

The following   well-known fact will be needed later.
\begin{lemma} \label{lem:comp_function_SI1O} Let $f\colon \, [0,1]\to \RR$ be computable. Then the sets  $\{x\colon \, f(x) < p\}$ and $\{x\colon \, f(x) >  p\}$ are $\SI 1$ sets,  uniformly in a rational $p$. \end{lemma}
\begin{proof} We verify the fact for sets of the form $\{x\colon \, f(x) < p\}$, the other case being symmetric. By (a$'$) in Subsection~\ref{ss:compfunctions}, $f$ is uniformly computable on the rationals in $[0,1]$. Suppose $(q_k) \sN k$ is a computable Cauchy name for a real $y$. Then $y<  s \LR \ex k \, [q_k < s - \tp{-k}]$.  So we can uniformly in a rational $s$ enumerate the  set of rationals $t$ such that $y= f(t)< s$. 
	
Now let $h\colon \NN \to \NN $ be a function showing the  effective uniform continuity  of $f$ in the sense of  (b) of Subsection~\ref{ss:compfunctions}. 	 To verify that  $\+S = \{x\colon \, f(x) < p\}$  is $\SI 1$, we have to show that  $\+S= \bigcup_k A_k$ where  $(A_k)\sN k$ is a computable sequence of open intervals with dyadic rational endpoints.

To define this sequence,  for each $n$ in parallel, do the following. Let $s_n= p- \tp{-n}$ and  $\delta_n = \tp{-h(n)}$. When a dyadic  rational  $t $  such that $f(t) <  s_n$ is enumerated, add to the sequence    $(A_k)\sN k$   the open interval $[0,1] \cap (t-\delta_n, t+ \delta_n)$. 

Clearly  $\bigcup_k A_k  \sub \+ S$. For the converse inclusion, given $x\in \+S$, choose $n$ such that $f(x)+ 2\cdot \tp{-n} < p$, and choose a rational $t$ such that $x$ is in the open interval $[0,1] \cap  (t-\delta_n, t+ \delta_n)$. Then $f(t) < s_n$, so this interval is added to the sequence. 
\end{proof}

Actually there is   uniformity at a higher level:  effectively in an index for~$f$, one can obtain an index for the function mapping $p$ to an index for the  $\SI 1$ set $\{x\colon \, f(x) < p\}$. 
 
\subsection{Some notation and facts on differentiability}   \label{ss:diff_prelims}
  
   Unless otherwise mentioned,   functions will have   a domain contained in the unit interval. 

For  a   function $f$,  the \emph{slope} at a pair $a,b$  of distinct reals  in its domain  is 
   
   \[ S_f(a,b) = \frac{f(a)-f(b)}{a-b}.\]
   Clearly $S_f(a,b) = S_f(b,a)$. 
   If $A$ is a nontrivial interval with endpoints $a,b$, we also write $S_f(A)$ for $ S_f(a,b)$.
   
   Recall that if $z$ is in the domain of $f$  and the domain is dense around $z$, then  \begin{eqnarray*} \ol D f(z)  & = & \limsup_{h\ria 0} S_f(z, z+h) \\
   \ul D f(z) & = & \liminf_{h\ria 0} S_f(z, z+h)\end{eqnarray*} 
   
   Note that we allow the values $\pm \infty$. 
By the definition, a   function  $f$ is    differentiable at $z$ if $\ul D f(z) = \ol D f(z)$ and this value is finite.  

For  $a < x < b$   we have  
\begin{equation} \label{eqn:weighted_average_slopes} S_f(a,b) = \frac{x-a}{b-a}S_f(a,x)+ \frac{b-x}{b-a}S_f(x,b). \end{equation}
	This implies the following:

\begin{fact} \label{fact:slopes} Let $a < x < b$. Then \[\min \{ S_f(a,x), S_f(x,b)\} \le S_f(a,b) \le \max \{ S_f(a,x), S_f(x,b)\}.\]  \end{fact}

Consider  a set $V \sub \RR$ that  is  dense in  $ [0,1]$. If $V$ is contained in the domain of a function  $f$,  we let 
\begin{eqnarray*}    D^Vf (x) & = &\lim_{h \ria 0+} \sup   \{S_f(a,b)  \colon \\ && \ \ \  a, b \in V\cap [0,1]   \lland  \, a\le x \le b \lland\, 0 <  b-a\le h \}\\ 
    D_V f (x) & = & \lim_{h \ria 0+} \inf  \{S_f(a,b)  \colon  \\ && \ \ \  a, b \in V\cap [0,1]   \lland  \, a\le x \le b \lland\, 0 <  b-a\le h\}. \end{eqnarray*}  

If $f(z)$ is    defined, then  Fact~\ref{fact:slopes}  implies that 

 \begin{equation} \label{eqn:DfDV}  \ul Df(z) \le D_V f(z) \le D^V f(z) \le \ol Df(z).   \end{equation}

The \emph{middle third} of an   interval $(a,a+d)$, where $0<d$, is the closed interval 
  $[a + d/3 ,   a+ d \cdot 2/3]$.   The following  lemma will be used in the proof of the main Theorem~\ref{thm:CRdiff}.     It implies that if a function $f$ is not differentiable at $z$, then this fact is witnessed on intervals that contain $z$ in their middle third. 
\begin{lemma} \label{lem:mid3}
Suppose  that $f\colon \, [0,1] \to \mathbb R$ is continuous at $z$. For any $h> 0$ let
 \bc $\+ I_h = \{(a,b)\colon 0<b-a <  h \, \lland \text{\rm  $z$ is in the middle third of }(a,b)\}$. \ec  
 Suppose that  \bc $v:=\lim_{h\to 0} \sup \{S_f(a,b)\colon (a,b)\in \+ I_h\} =\lim_{h\to 0} \inf \{S_f(a,b)\colon (a,b)\in \+ I_h\}$. \ec Then $f'(z)=v$.\end{lemma}
\begin{proof}
The continuity of $f$ at $z$ implies that $f'(z)$ equals  the limit of $S_f(a,b)$ over all open intervals $(a,b)$ that contain $z$, as $b-a\to 0$. 
Take $h$ and $t<s$ such that $t<S_f(a,b)<s$ for all $(a,b)\in \+ I_h$. Consider an interval $(c,d)$ containing $z$ such that $d-c<h/3$. We will prove that \bc $5t-4s<S_f(c,d)<5s-4t$.  \ec   Note that we can take $t$ and $s$ to approach $v$ as $h\to 0$, in which case both $5t-4s$ and $5s-4t$ also approach $v$. This implies that $f'(z)=v$.

Assume, without loss of generality, that $z$ is closer to $c$ than to $d$.    The idea is to define a sequence of  intervals  $(a_n,b_n)$,  $(a_{n+1},b_n)$  in $\+ I_h$   of increasing length.  We start with $a_0=c$ and $b_0 =  a_0 + 3(z-c)$.   The real   $z$ is the least in the  middle third of $(a_n,b_n)$, and the greatest in  the  middle third of $(a_{n+1},b_n)$.     The intervals ``see-saw''  around $z$ until  we reach  $N$ such that $b_N<d\leq b_{N+1}$.   We first over-estimate $f(d)-f(c) $   by $f(d)-f(a_{N+1})$, then subtract an   over-correction $f(b_N)-f(a_{N+1})$,  then add a second correction $f(b_N)-f(a_N)$, and so on, until we add  $f(b_0)-f(a_0)$ and get the right value. The   terms we add are  bounded from above by the length of the corresponding  interval times $s$.  The    terms we subtract are  bounded from below by the length of the interval times $t$.  This will show that $S_f(c,d)<5s-4t$.
  
For the details,   let $\delta = z-c$.   We   let $a_n = z-2^{2n}\delta$ and $b_n=z+2^{2n+1}\delta$ for all $n\in\omega$.   We may assume that $b_0<d$, because otherwise $z$ would be in the middle third of $(c,d)$.    Note that $z$ is in the middle third of $(a_{N+1},d)$ because $b_{N+1} \ge d$. This interval is the longest we will consider. Note that  \bc  $d-a_{N+1}=(d-z)+(z-a_{N+1})<(d-c)+2^{2N+2}\delta<3(d-c)\leq h$, \ec 
because  $2^{2N+1}\delta = b_N-z < d-c$.  Therefore, all of the intervals that occur in the  first four  lines of the following  estimates  are in $\+ I_h$. We have
\begin{eqnarray}  
f(d)-f(c) & =& (f(d)-f(a_{N+1})) - (f(b_N)-f(a_{N+1})) + (f(b_N)-f(a_N))  \nonumber \\  
		&&  - \cdots  + \cdots  - (f(b_0)-f(a_1)) + (f(b_0)-f(a_0))      \nonumber  \\
 		&\le & s(d-a_{N+1}) - t(b_N-a_{N+1}) + s(b_N-a_N)  \nonumber \\ 
		&&  - \cdots  + \cdots   - t(b_0-a_1) + s(b_0-a_0) \label{line1}  \\
		& =&  s(d-c) + s(a_0-a_{N+1}) - t(2^{2N+3}-2)\delta + s(2^{2N+2}-1)\delta \nonumber  \\
		& =&  s(d-c) + (s-t)(2^{2N+3}-2)\delta   \nonumber  \\ 
		&<&  s(d-c) + 4(s-t)2^{2N+1}\delta  \nonumber  \\
		& <&  s(d-c) + 4(s-t)(d-c).   \nonumber
\end{eqnarray}
This proves that $S_f(c,d)<5s-4t$. The lower bound   $5t-4s<S_f(c,d)$ is obtained in an  analogous way.  \end{proof}

\subsection{Binary expansions} 
	 \label{sub:binary_expansions}

	By a \emph{binary expansion} of a real $x \in [0,1)$ we will always mean the one with infinitely many 0s. Co-infinite sets of natural numbers are often  identified with reals in $[0,1)$ via the binary expansion. In this way, the product measure on Cantor space $\cantor$ is turned into the \emph{uniform} (Lebesgue) measure on $[0,1]$.

%

\section{Computable randomness} 

\label{s:CRand_intro}


\subsection{Background on  computable randomness}
\label{ss:MG}   \label{ss:randomness_notions}

%
%
%
  
  Schnorr  \cite{Schnorr:75}  proposed \emph{computable betting strategies} as tests for randomness. They are certain    computable functions~$M$ from $\strcantor $ to the non-negative reals.    
  Let $Z $ be an infinite sequence of bits, and let $Z \uhr n$ denote the first $n$ bits. When the player has seen $\sss= Z \uhr n$, she can make a bet~$q$, where $0 \le q \le M(\sss)$, on  what  the next bit $Z(n) $ is.     If she is right, she gets $q$. Otherwise she loses $q$. 
  The formal concept corresponding to  a betting strategy is   the following.

  \begin{deff} \label{df:MG}  A \emph{martingale} is  a function $2^{< \omega} \ria \R^+_0$ such that the fairness condition 
   \begin{equation} \label{eqn:fairness}  M(\sss 0) + M(\sss 1) = 2 M(\sss) \end{equation} 
   holds    for each string $\sss$.  
$M$ \emph{succeeds}  on a sequence of bits   $Z$ if  $M(Z\uhr n)$ is unbounded.   \end{deff}

Recall from Subsection~\ref{ss:CompReals} that a  real number $x$ is called computable if  there is a computable Cauchy name  $(q_n)\sN n$  of rationals such  that $|x-q_n | \le \tp{-n}$ for each $n$. A martingale $M \colon \, 2^{< \omega} \ria \R^+_0$ is called \emph{computable} if $M(\sss)$ is a computable real  uniformly in  a string  $\sss$.

\begin{deff}  \label{df:CR} An infinite  sequence of bits  $Z$ is called \emph{computably random} if  no computable martingale succeeds  on $Z$.  
A real $z \in [0,1) $ is called  \emph{computably random} if its binary expansion is computably random. \end{deff} 
In fact,  it suffices to require that no  rational-valued martingale succeeds on the binary expansion of $z$ (\cite{Schnorr:75}, also see \cite[7.3.8]{Nies:book}).
We mention some facts about computable randomness. For details, definitions, references and  proofs  see for instance \cite[Ch.\ 7]{Nies:book} or \cite{Downey.Hirschfeldt:book}.

Computable randomness    lies strictly in between \ML\ and Schnorr randomness. Computably random sets  can have a  very slowly growing   initial segment complexity, e.g., $K(Z\uhr n) \le^+ 2 \log n$.   A left-c.e.\  computably random set can be   Turing incomplete. In fact,  such a set  exists   in each high c.e.\ degree.  
  There is a characterization  of computable randomness by Downey and Griffiths \cite{Downey.Griffiths.ea:04}   in terms of special \ML{} tests called ``computably graded tests'', and a  characterization  by Day \cite{Day:09}  via the growth of  initial segment complexity measured  in terms of so-called  	``quick process machines''.

\subsection{The savings property}

   \label{ss:savings}

 \begin{deff}\label{df:Savings} We say that a martingale $M$ has the {\it savings property}  if  $M(\rho) \ge  M(\sss) -2$  for any strings $\sss, \rho$ such that $\rho \succeq \sss$. \end{deff}
The following is well-known (see \cite{Downey.Hirschfeldt:book} or \cite[7.1.14]{Nies:book}).

 \begin{prop}  \label{prop:Savings} For each computable  martingale $L$ there is a  computable martingale $M  $ with the savings property  that  succeeds on the same sequences as $L$.   \end{prop}  
  \begin{proof}  We may assume that $L (\sss ) >0 $ for each $\sss\in \strcantor$, and $L(\estring) <1$.  As mentioned above, we may also assume that $L$ is rational valued  by a result of Schnorr (see  \cite[Prop.\  7.3.8]{Nies:book}). We let $M=G+E$, where  $G(\sss)\in \NN$ is the  balance of the  ``savings account'',  and~$E(\sss)$ is the  balance of   the ``checking account'' at $\sss$. The function $E$  is a supermartingale (see \cite[Section  7.2]{Nies:book}) bounded by~$2$.   It uses    the same betting factors $L(\rho  \ape b)/L(\rho ) $ as~$L$ for  a string~$\rho$ and $b \in \twoset$, but in between subsequent  bets it   may transfer capital to the savings account.

 For each string $\rho$, whenever  $b \in \twoset$ and the betting   results in a value  $v >1$ at $\rho \ape b$,   we transfer $1$  from the checking to  the savings account, defining  $G(\rho \ape b) = G(\rho)+1$;  in this   case $E$    has the capital $v -1$  at the  string  $  \rho  \ape  b$.    If $v \le 1$  we let  $ E(\rho  \ape  b) =v$ and  $G(\rho  \ape  b) = G(\rho)$.
   
 $M= G+E$ has the savings property because,   if $\rho  \succeq \sss$ are strings, then \bc $M(\rho) -  M(\sss) \ge E(\rho) - E(\sss) \ge  -2$. \ec  If $L$ succeeds on $Z$  then $\lim_n G(Z \uhr n) = \infty$, whence $\lim_n M(Z \uhr n) = \infty$. Since  the rational valued martingale $L$ is computable   so is $M$:  its values are  obtained by applying  arithmetical operations to the values of $L$, which are given effectively by computable  Cauchy names for reals, and arithmetical operations preserve this.  \end{proof}

In general,  if  $M$  is a martingale, then  $M(\sss) \le \tp {\sssl} M(\estring)$  for each string $\sss$. If  $M$ has the savings property,  then   in fact  
 
 \begin{equation} \label{eq:bdMG} M(\sss) \le 2 \sssl+ M(\estring). \end{equation}
 For otherwise,  there is $\tau \ape i \preceq \sss$  for some $i \in \twoset$ such that $M(\tau  \ape  i ) > M(\tau) +2$, whence $M(\tau \ape  (1-i)) < M(\tau) -2$.

\subsection{A correspondence between martingales and nondecreasing   functions}  
\label{ss:corrMGfunctions} 
For  a string $\sss \in \strcantor$ we will write   \[[\sss) = [0.\sss, 0.\sss + \tp{-\sssl}); \]
we use the notation  $[\sss]$ either to denote the cone $\{ X \colon \, X \succ \sss\}$ in Cantor space, or the corresponding closed subinterval of $[0,1]$.
 
 Each  martingale $M$ determines  a  measure  on the algebra of clopen sets by  
  assigning $[\sss]$ the value $ \tp{-\sssl} M(\sss) $.   Via Carath\'eodory's  extension  theorem this measure  can be extended to  a  Borel measure  on Cantor Space. We say that $M$ is \emph{atomless} if this measure is atomless, i.e., has no point masses.  Note that, by (\ref{eq:bdMG}) every martingale with the savings property is atomless.   If the measure  is atomless,  via the binary expansion of reals  (see Subsection~\ref{sub:binary_expansions}) we can  also  view it  as a  Borel measure  $\mu_M$ on  $[0,1]$.  Thus,  $\mu_M$ is determined by the condition  
  %
  %
  \begin{equation} \label{eqn:MartSlope} \mu_M [\sss) = \tp{-\sssl} M(\sss). 
\end{equation}
  We use the equality (\ref{eqn:MartSlope}) above  to  establish a relationship  between   atomless martingales and  nondecreasing continuous  functions. (This  is essentially a special case of the correspondence between signed Borel measures on $[0,1]$ and  left-continuous functions of bounded variation  that vanish at $0$. See, for instance, \cite[8.14]{Rudin:74}. We will need the notation and details of this correspondence later on.)
 
 \vsp
 
 \n  {\it Atomless martingales to nondecreasing continuous   functions on $[0,1]$.}     Given an atomless  martingale~$M$, let $\cdf (M)$  be  the cumulative distribution function of the associated measure. That is, 
  \[ \cdf(M)(x) = \mu_M[0,x).  \] Then  $\cdf(M)$ is nondecreasing and  continuous since the measure is atomless. 
Hence it is   determined  by its values on the rationals.

  \vsp
  
 \n We   let $I_\Q = [0,1] \cap \QQ$.
  
   \n  {\it Nondecreasing functions  with domain containing  $I_\Q$ to  martingales.}    Suppose  $f$ is      
a nondecreasing function with a domain containing $I_\Q$. We  will  write  
   \begin{equation} \label{eqn:Martf}  \Mart(f)(\sss) =  S_f(\sss)  =( f(0.\sss + \tp{-\sssl}) -f(0.\sss)   )/ \tp{-\sssl}. \end{equation}
Let  $M = \Mart(f)$. We have, for instance, $M(10)= S_f(\frac  1 2, \frac 3 4)$, and $M(11) = S_f(\frac 3 4, 1)$.  That $M$ is a martingale follows from the   averaging condition  on slopes in  (\ref{eqn:weighted_average_slopes}). 
To see that  $M$ is a martingale, fix $\sss$.    Let $a = 0.\sss, x= 0.\sss + \tp{-\sssl-1}$, and $b= 0.\sss + \tp{-\sssl}$. By the   averaging condition  on slopes in  (\ref{eqn:weighted_average_slopes}) we have 
\[ M(\sss) = S_f(a,b) = S_f(a,x)/2 + S_f(x,b)/2= M(\sss 0) /2 + M(\sss 1)/2.\]
 
\begin{fact}\label{fa:correspondence} The transformations defined  above induce a correspondence between atomless martingales and nondecreasing   continuous functions on $[0,1]$    that vanish at $0$. In particular:
\bi \itone  Let $M$ be an atomless  martingale. Then  $\Mart(\cdf(M)) =M$. 

\ittwo  Let $f$ be a nondecreasing   continuous function  on $[0,1]$  such that  
 $f(0) = 0$. Then
  $\cdf(\Mart(f)) = f$. \ei \end{fact}

\begin{proof} 
%
 (i) is clear. For (ii), let $M = \Mart(f)$.  Let $\mu$ be the measure on $[0,1]$ such that $\mu[0,x) = f(x)$ for each $x$. Then $M(\sss) =  \tp{\sssl} \mu[\sss)$ for each~$\sss$. Hence $\mu_M = \mu$ and $\cdf (M) (x) = \mu_M[0,x)= f(x)$ for each~$x$.    \end{proof}

Recall the definition of $D_V(z)$ from  Subsection~\ref{ss:diff_prelims}.  

\begin{thm} \label{thm:MonSuc} Suppose $M$  is a  martingale with  the savings property (see Subsection~\ref{ss:MG}).  Let $g= \cdf(M)$.     Suppose  $z \in [0,1]$  is not a dyadic rational. Then the following are equivalent:

\bi \itone   $M$ succeeds on the binary expansion of $z$.

\ittwo                   $\underline D g (z) = \infty$.  
\itthree $D_\Q g(z)= \infty$.
\ei 
\end{thm}
The proof will show that the  implications (ii)$\to$(iii)$\to$(i)   do not rely on  the hypothesis that $M$ has the savings property. However,  we always  need the weaker property that $M$ is atomless to ensure that $\cdf (M)$ is defined. 
\begin{proof}  Note that, since  $z  \in [0,1)$   is not a dyadic rational, its      binary expansion~$Z$  is unique.

\n (ii) $\to$ (iii).  This is immediate because, by (\ref{eqn:DfDV}) in Subsection~\ref{ss:diff_prelims}, we have $ \underline D g (z) \le D_\Q g(z)$.


\vsps
 
 \n (iii)$\to$(i). Given $c> 0$, choose $n$ such that   $S_g(p,q)\ge c$ whenever $p,q$  are rationals, $p \le z \le q$, and  $q-p \le \tp{-n}$. Let $\sss = Z\uhr n$. Then we have $z \in [\sss]$, and the length of this interval is   $\tp{-n}$.  Hence $M(\sss) \ge c$.

\vsps

\n (i) $\to$ (ii). We     show that for each   $r \in \NN$ there is $\eps >0$ such that $0< |h | < \eps$ implies $(g(z+h) - g(z))/h  \ge  r$. This implies  that   $\underline D g(z) = \infty$.

Note that the binary expansion $Z$ of $z$ has   infinitely many 0s and infinitely many 1s.
 Since $M$ has the savings property, there is $i \in \NN$ such that $Z(i) =0 $, $Z(i+1) = 1$,  and for $\rho = Z\uhr {i}$,  we have $\fao \tau M(\rho  \tau ) \ge r$. Let $j > i$ be  least such that $Z(j) =0$. Let $\eps = \tp{-j-1}$. If $0< |h| < \eps $ then the binary expansion of $z+h$ extends $\rho$. If $h>0$, this is because $z+ \tp{-j-1} < 0. \rho1$. If $h< 0$, then adding  $h$ to  $z$ can at worst  change the bit $Z(i+1)$ from $1$ to $0$. 

%

For $V \sub \strcantor$ let $\Opcl V$ denote the set of infinite sequences of bits extending a string in $V$. Let $W \sub \strcantor$ be a prefix free set of strings such that $\Opcl W$ is identified with the open interval $ (z,z+h)$ in case $h > 0$, and 
$\Opcl W $ is identified with $ (z+h,z)$ in case $h < 0$.  All the strings in $W$ extend $\rho$. So we have in case $h>0$, 
\[ g(z+h) -g(z)= \mu_M (z, z+h) = \sum_{\sss \in W} M(\sss) \tp{-\sssl} \ge r \sum_{\sss \in W} \tp{-\sssl} = rh, \]
 and in case $h<0$
\[ g(z) -g(z+h) = \mu_M(z+h,z)= \sum_{\sss \in W} M(\sss) \tp{-\sssl} \ge r \sum_{\sss \in W} \tp{-\sssl} = -rh. \]
In either case we have  $(g(z+h) - g(z) )/h  \ge r$.   \end{proof}


%

\subsection{Computable randomness is  base-invariant} 

\label{ss:CRand_base_invariant}
We give a first application of the analytical view of algorithmic randomness. 

 If the definition of a  randomness notion for Cantor space  is based on measure, it  can be transferred right away to the reals in $[0,1]$ by the correspondence in Subsection~\ref{sub:binary_expansions}. 

 Among the notions in the hierarchy mentioned in the introduction, computable randomness is the only one not directly defined in terms of  measure. 
We argue that computable randomness of a real is independent of the choice of base for  expansion. We will use that the condition (iii) in Theorem~\ref{thm:MonSuc} is base-independent. First, we give the relevant  definitions.

Let $k \ge 2$.
A \emph{martingale for base $k$} is  a function \bc $M \colon \, \{0, \ldots, k-1\}^{< \omega} \ria \R^+_0$  \ec  with the fairness condition $ \sum_{i=0}^{k-1}  (  M(\sss i) -M(\sss) )   = 0 $, or  equivalently, 
$$ \sum_{i=0}^{k-1} M(\sss i) = k M(\sss).$$
(An example is repeatedly playing  a simple type of  lottery, where $k$ is the number of possible draws. The player has  seen a string $\sss$ of draws. She bets an amount $q \le M(\sss)$   on a     certain draw; if  she is   right she gets $(k-1) \cdot q $, otherwise she loses $q$.)

The topics in    Subsections~\ref{ss:savings} and~\ref{ss:corrMGfunctions}  can be developed more generally for martingales $M$  in base $k$. Such a martingale induces  a measure $\mu_M$ on $k^\omega$  via $\mu_M([\sss])= M(\sss) k^{-\sssl}$.  As before, we call $M$ atomless if $\mu_M$ is atomless as a measure. The remarks on the savings property after Definition~\ref{df:Savings} remain true; the condition   (\ref{eq:bdMG}) turns into $ M(\sss) \le 2 (k-1) \sssl+ M(\estring)$.   We have a transformation  $\cdf$ turning an atomless   martingale in base $k$ into a nondecreasing continuous  function on $[0,1]$ vanishing at $0$, namely,  the distribution function of $\mu_M$. There is  an inverse transformation $\Mart^k$ turning such a   function $f$ into a martingale in base $k$ via 
\bc $\Mart^k(f)(\sss) = S_f(0.\sss ,0. \sss + k^{-\sssl})$. \ec 

 We call a sequence  $Z$  of numbers in  $\{0, \ldots, k-1\} $   \emph{computably random in base $k$} if  no computable martingale in base $k$ succeeds  on $Z$. Let us temporarily say that  a real  $z \in [0,1)$  is \emph{computably random in base $k$} if its base $k$ expansion  (with infinitely many entries  different from $ k-1$)  is computably random in base $k$.

\begin{thm} \label{thm:CRbaseinv} Let $z\in [0,1)$. Let $k, r \ge 2$ be  natural numbers.  Then 

\n 
 $z$ is computably random in base $k$ $\RA $ $z$ is computably random in base $r$.   \end{thm}

\begin{proof}  We may assume $z$ is irrational. 
 Let $Z $ be the base $k$ expansion, and  let $Y$ be the base $r$ expansion of $z$. Suppose $Y$ is not computably random in base~$r$. Then  some computable martingale $M$ in base $r$ with the savings property succeeds on $Y$. By    (\ref{eq:bdMG})    for base $r$  we have $M(\sss) \le 2 (r-1)\sssl + O(1)$,  whence $M$ is atomless. Hence $\mu_M$ is defined and the associated distribution function 
   $f= \cdf(M)$ is continuous. Clearly,  $f(q)$ is uniformly computable for   any rational    $q \in [0,1]$ of the form $ir^{-n}$, $ i \in \NN$. Hence, by Proposition~\ref{prop:monotonic computable}, $f$ is computable.  Therefore the martingale   in base $k$ corresponding to $f$, namely   $N=\Mart^k(f)$,  is  atomless and computable. 

The proof of (i)$\to$(ii) in Theorem~\ref{thm:MonSuc}  works for  base $r$: replace 2 by $r$, and replace the digits 0,1 by    digits $b<c<r$ that both occur infinitely often in the $r$-ary expansion of $z$ (they exist because $z$ is irrational).  So, since $M$ has the savings property, we have $\ul D f(z) = \infty$. Note that  $f= \cdf(N)$ by Fact~\ref{fa:correspondence} in base $k$. Hence   by  (ii)$\to$(i) of  the same
Theorem~\ref{thm:MonSuc},  but  for base $k$, the computable martingale $N$ succeeds on~$Z$. 
\end{proof}

\Com{(this will go into some paper on poly time rdness) For the base invariance of polynomial time  randomness, suppose $M$ is a polynomial  time base $r$ martingale succeeding on $Y$. We use the savings property of $M$ to verify that for $x< y$ the function $f$ satisfies  the ``almost-Lipschitz condition''
\[ f(y)- f(x)  = -\log (y-x) O(y-x) . \]
Let $n\in \NN$ be least such that $ r^{-n} <  y-x$, and let the    $p$ be the least rational of the form $i r^{-n}$  such that $x \le p+r^{-n}$. Then $[x,y] \sub [p, p+ (r+1)  \cdot r^{-n }]$. Hence by consequence (\ref{eq:bdMG}) of  the savings property we have $f(p + (i+1) r^{-n}) - f(p+i r^{-n})  \le (2n + M(\estring))r^{-n}$ for each $i=0,\ldots, r+1$, whence 

\bc $f(y)- f(x) \le   f(p+ (r+1)  \cdot r^{-n}) - f(p) \le ((r+1) 2 n + (r+1) M(\estring))r^{-n}$. \ec
Since  $r^{-n} <  y-x$  and $n \sim -\log (y-x)$,  this yields the desired  bound.   

This can now be used to show that   $f$ is poly time computable on the rationals (given rational  $p= z/n$ and precision $i$ in binary we can in polynomial  time compute rational $q $ such that $|f(p)-q| \le \tp{-i}$) , and hence $N=\Mart^k(f)$ is a polynomial  time martingale succeeding on $Z$.}

%

\section{Computable randomness and differentiability}

\label{s:CRandDiff}
   We characterize computable randomness in terms of differentiability. 

\begin{thm} \label{thm:CRdiff} Let $z\in [0,1)$. Then the following are equivalent:
\bi \itone $z$  is computably random.  

\ittwo    Each computable  nondecreasing  function   $f \colon \, [0,1] \ria \mathbb R$ is differentiable at~$z$. 

\itthree  Each computable  nondecreasing  function   $g \colon \, [0,1] \ria \mathbb R$ satisfies 

\n  $\ol D g(z) < \infty$.

\itfour  Each computable  nondecreasing  function   $g \colon \, [0,1] \ria \mathbb R$ satisfies 

\n  $\ul D g(z) < \infty$.  \ei  \end{thm}

\begin{proof} 
 The implications  (ii)$\to$(iii)$\to$(iv) are trivial. For the implication (iv)$\to$(i),  suppose that $z$    is not computably random. 

If $z$ is    rational, we can let $g(x) = 1- \sqrt{z-x}$ for $x \le z$ and $g(x) =1 $ for $x>z$. Clearly $g$ is nondecreasing and $\ul D g (z)= \infty$. Since $z$ is rational, $g$ is uniformly computable on the rationals in $[0,1]$. Hence  $g$ is computable by Proposition~\ref{prop:monotonic computable}.

Now suppose  that   $z$ is irrational. Let the bit sequence  $Z$ correspond to  the binary expansion of~$z$.   By Prop.\ \ref{prop:Savings}, there is a computable martingale~$M$ with the savings property such that   $ \lim_n M(Z \uhr n)=\infty$.   Let $g = \cdf(M)$. Then     $\ul D g(z) = \infty$ by  Theorem~\ref{thm:MonSuc}.  
 By the savings property of $M$,   the associated distribution function $g= \cdf(M)$ is continuous. 
 Clearly $g(q)$ is uniformly computable on the dyadic rationals in $[0,1]$. Then,  once again  by  Proposition~\ref{prop:monotonic computable},   we may conclude that  $g$ is computable.

 It remains to prove the implication (i)$\to$(ii).
 
We actually   prove (i)$\to$(iii)$\to$(ii). We begin with  (iii)$\to$(ii). A naive approach  runs into trouble, which  motivates the  algebraic Lemma~\ref{lem:rat_intervals} below.  Thereafter, we will obtain     (i)$\to$(iii) by another application  of that lemma.


\subsection{Proof of (iii)$\to$(ii)}

\subsubsection{Bettings on rational intervals}  For the rest of this proof,   intervals will be  closed  with distinct  rational endpoints unless otherwise mentioned. If $A = [a,b]$ we write $|A|$ for the length $b-a$.   To say that intervals are \emph{disjoint} means they are disjoint on $\R\setminus\QQ$. 
A \emph{basic dyadic interval} has the form $[i\tp{-n}, (i+1)\tp{-n}]$ for some $i \in \ZZ, n \in \NN$.

  We  prove  the contraposition  $\neg$(ii)$\ria \neg$(iii).  Suppose a  computable  nondecreasing function  $f$ is not differentiable at~$z$.   We will eventually define a computable nondecreasing function $g$ such that  $\ol Dg(z) = \infty$.  We may assume $f$ is   increasing after replacing $f$ by the function  $x \mapsto f(x) +x$.  If $\ul Df(z) = \infty$ we are done by letting $g= f$. Otherwise, 
we have   

\[ 0 \le \ul Df(z) < \ol Df(z).\]

 The nondecreasing computable function $g$ is defined in conjunction  with  a betting  strategy $\Gamma$. Instead of betting on  strings, the strategy   bets  on nodes in a   tree   of rational intervals~$A$. The root is $[0,1]$, and the tree is  ordered by  reverse inclusion. This strategy $\Gamma$  proceeds from an interval $A$   to   sub-intervals $A_k$ which are its successors on the tree.  It   maps these intervals to non-negative reals representing the   capital at that interval.   If the tree consists of the basic dyadic sub-intervals of $[0,1]$, we have essentially the same type of betting strategy as before. However,  it will be necessary to consider a more complicated tree where  nodes have infinitely many successors.
 
 We define the nondecreasing function  $g$ in such a way that  the  current capital  at   a node $A= [a,b]$  is the slope:
\begin{equation}  \label{eqn:Gamma}  \Gamma(A) = S_g(a,b) = \frac{g(b)- g(a)}{b-a}. \end{equation}
Thus, initially we define $g$ only on the endpoints of intervals in the tree, which will form a dense    sequence of rationals  in $[0,1]$ with an effective listing. Thereafter we will use Proposition~\ref{prop:monotonic computable} to extend $g$ to all reals in the  unit interval. 

%


%

\subsubsection{The Doob strategy} One idea in  our proof is taken from the proof of the fact   that  $\lim_n M(Z\uhr n)$ exists  for each computably random sequence~$Z$ and each computable martingale $M$:  otherwise, there are rationals $\beta,  \gamma$ such that
 \bc $\liminf_n M(Z \uhr n) < \beta < \gamma < \limsup_n  M(Z\uhr n)$.  \ec 
 In this case one defines a new computable betting strategy $G$ on  strings that succeeds on $Z$. On each string,  $G$  is  either in the betting state,  or in the non-betting state. Initially it is in the  betting state. In the betting state $G$ bets with the same factors  as $M$   (i.e.,  $G(\sss a)/G(\sss)=M(\sss a)/M(\sss)$ for the current string $\sss$ and each $a \in \twoset$), until $M$'s capital exceeds $\gamma$. From then on, $G$ does not bet until 
$M$'s capital is  below~$\beta$. On the initial segments of~$Z$,  the strategy $G$ goes through  infinitely many state changes; each time it  returns  to the non-betting state,  it has multiplied its capital by $\gamma/\beta$. Note that this is an effective version of    the technique used to prove  the first Doob martingale convergence theorem.

Recall that if  $A = [x,y]$,  for the slope of $f$ we use the shorthand $S_f(A) = S_f(x,y)$. 
 Given $z \in [0,1]-\QQ$,  let $A_n$ be the basic  dyadic interval of length $\tp{-n}$ containing $z$. Naively, one could hope that   our case  assumption $\ul Df(z) < \ol Df(z)$    becomes apparent on these  basic  dyadic intervals:   \bc $\liminf_n S_f(A_n)  <\beta  < \gamma <  \limsup_n S_f(A_n)$  \ec for some rationals $\beta < \gamma$. In this case, we   may  carry out the Doob strategy for  the martingale $M$   given by  $M(\sss)= S_f(\sss)$ as defined in \eqref{eqn:Martf},   and view it as a betting strategy   on nodes in the tree of basic  dyadic intervals.     

Unfortunately, this scenario is too simple. If we allow $z=0$ then this already becomes  apparent via the   computable increasing function $f(x)= x \sin (2 \pi \log_2 x) + 10x$, because $f(x)=10x$ for each $x$ of the form $\tp{-n}$, but $9 = \ul Df(0) < \ol Df(0) = 11$.  If $z$ is a computable irrational,   the   function indicated in Figure~\ref{fig:figure2}    satisfies   $\ul Df(z) <   \ol Df(z)$, but $S_f(A_n) =1$ for each $n$.

 \begin{figure}[htbp] 
   \scalebox{0.4}{\includegraphics{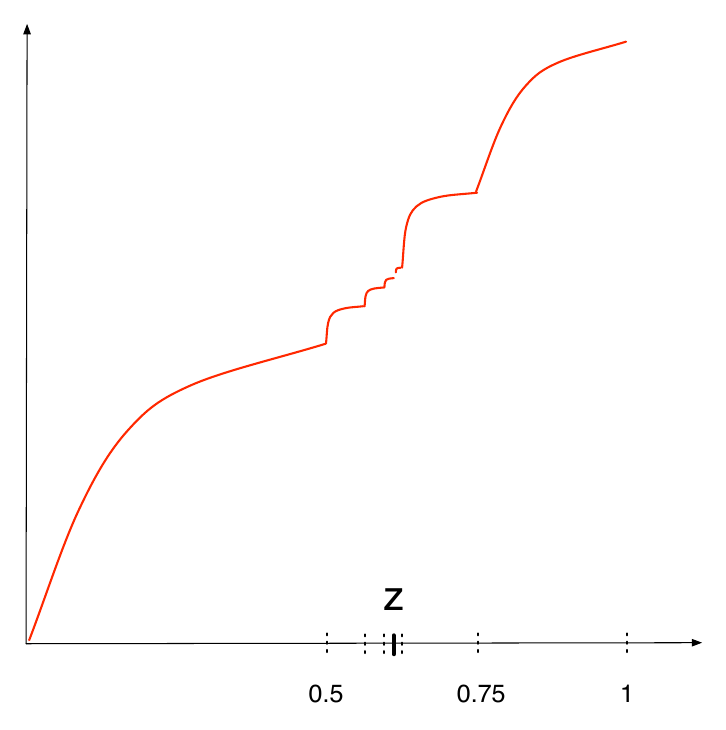}} 
   \caption{A function that is only dyadically  differentiable at $z$}
   \label{fig:figure2}
\end{figure}

We now describe how to deal with the general situation. Our main technical concept is the following. For $p,q\in\Q$, $p>0$,  we say that an interval is a  \emph{$(p,q)$--interval} if it is the image of a basic dyadic  interval    under the affine transformation $y \mapsto py+q$. Thus,   a $(p,q)$--interval   has the form 

\[ [p i \tp{-n}+q, p (i+1)\tp{-n}+ q  ] \] 
for some $i \in \ZZ, n \in \NN$.

We will show in Lemma~\ref{lem:ratspq} that there are rationals $p,q$ and $r,s$ such that 
\begin{equation} \label{eqn:pqrs}  \liminf_{\begin{subarray}{c}   |A| \to 0   \\   A \ttext{is a} (r,s)-\text{interval} \\ z \in A  \end{subarray} } S_f(A)  \, \,  <  \limsup_{  \begin{subarray}{c}    |B| \to 0   \\   B \ttext{is an}  (p,q)-\text{interval} \\ z \in B  \end{subarray}}  S_f(B).\end{equation}
The strategy $\Gamma$ is in the  betting state on $(p,q)$ intervals in the tree of intervals, and in the non-betting state on $(r,s)$--intervals. For  each state, it proceeds exactly like the Doob strategy in  the corresponding state. In addition,   when $\Gamma$  switches state, the current interval is split  into intervals of the other type (usually, into infinitely many intervals).  Nonetheless, the other state takes effect immediately. So, in  the betting state, we have to immediately bet on all the components of  this (usually)  infinite splitting.

\subsubsection{The algebraic part.}

We derive Lemma~\ref{lem:ratspq} from an algebraic lemma. For a set  $L$      of rationals, an  interval is called an  \emph{$L$--interval}
if it is a     $(p,q)$--interval  for some  $p,q \in L$.

\begin{lemma} \label{lem:rat_intervals}
For each  rational $\alpha >1$, we can effectively determine a finite set $L$ of  rationals  in $[-1,1]$  such that for each interval $[x,y]$, $0 < x< y < 1$, there   are  $L$--intervals  $A,B$   as follows:

\begin{eqnarray*} [x,y] \subset A   & \land  & \frac{|A|}{y-x }< \alpha,   \\
				 B \subset [x,y]  & \land &  \frac{y-x} { |B|} < \alpha . \end{eqnarray*}
\end{lemma}
\n Informally, we can approximate the given interval from the outside  and from the inside by $L$--intervals within a ``precision factor'' of $\alpha$. We defer the   proof of the lemma    to  Subsection~\ref{ss:Lemmaproof}.

By the hypothesis that $f'(z)$ does not exist and Lemma~\ref{lem:mid3}, we can choose   rationals $\wt \beta < \wt  \gamma$ such that

 \begin{eqnarray*}  \wt \gamma  &   <  &   \lim_{h \ria 0} \sup \{ S_f(x,y)\colon \,  0 
 \le y-x \le h \lland  z  \in   (x,y)  \},  \\
\wt \beta  &   > &    \lim_{h \ria 0} \inf \{ S_f(x,y )\colon \,  0 
 \le y-x \le h \lland    z  \in  \text{  middle third of}  \  (x,y) \} \end{eqnarray*}

  Let   $\alpha,  \beta,  \gamma$ be rationals such that $1< \alpha < \frac 4 3 $ and 
\bc $   \wt  \beta \alpha <  \beta <       \gamma <    \wt  \gamma / \alpha$.  \ec 
\begin{lemma} \label{lem:ratspq} There are rationals   $p,  q$, $r,  s$, such that  $p, r > 0$ and 
\begin{eqnarray*}  \gamma  &   <  &   \lim_{h \ria 0} \sup \{ S_f(A)\colon \, A \ttext{is a} (p,q)\text{--interval}    \lland    |A| \le h  \lland  z \in A  \},  \\
\beta  &   > &    \lim_{h \ria 0} \inf \{ S_f(B)\colon \, B \ttext{is an} (r,s)\text{--interval}   \lland  |B| \le h    \lland  z \in B  \}. \end{eqnarray*}
  \end{lemma}

\begin{proof} Let $L$ be as in  Lemma~\ref{lem:rat_intervals}.  
For the first inequality, we  use the first  line in Lemma~\ref{lem:rat_intervals}. 


 Let   $h > 0$ be given. Choose  reals $x< y$, where $x \le  z \le y$,  such that  $y-x < h/ \alpha $ and $S_f(x,y) > \wt\gamma$. By Lemma~\ref{lem:rat_intervals} there is an $L$--interval $A= [u,v]$ such that $[x,y] \sub A $ and $|A|/ (y-x)  < \alpha$.   Then, since $f$ is nondecreasing and $v-u < \alpha (y-x)$, we have 
\[ S_f(A)  = \frac{f(v)-f(u) }{v-u}  \ge \frac{f(y) -f(x) }{ v-u} > \frac{f(y) -f(x) }{ (y-x)\alpha} >  \wt \gamma/ \alpha >  \gamma .\]
 Since $L$ is finite, we can now pick  a single pair of rationals $p,q  \in L$ which works for arbitrary small $h>0$,  as required.

For the second inequality, the argument is similar, based on  the second line in Lemma~\ref{lem:rat_intervals}.  However, we  also      need the condition on  middle thirds in the definition of $\wt \beta$, because when we replace an interval $[x,y]$ by a  subinterval $B$ that is an $(r,s)$-interval,  we   want  to ensure that  $z\in B$.

Given    $h > 0$, choose  reals $x< y$, where $x \le  z \le y$ and $z$ is in the middle third of $[x,y]$,  such that  $y-x < h/ \alpha $ and $S_f(x,y) <  \wt\beta$. By Lemma~\ref{lem:rat_intervals} there is an $L$--interval $B= [u,v]$ such that $B \sub [x,y]  $ and $  (y-x)/|B|   < \alpha$.   Since $\alpha < 4/3$  and $z$ is in the middle third of $[x,y]$, we have $z \in B$.  Similar to the estimates above, we have 
\[
S_f(B)     \le \frac{f(y) -f(x) }{ v-u} \le  \frac{f(y) -f(x) }{ (y-x)/ \alpha}  = \alpha S_f(x,y) < \alpha \wt \beta  < \beta.\qedhere
\]
\end{proof}

\subsubsection{Definition of $g$ on a dense set,  and the strategy  $\Gamma$} \label{sss:Defg}
In the following fix $p,q,r,s$ as in Lemma~\ref{lem:ratspq}. 
 Recall that we  plan to  define an infinitely branching tree of intervals, and that,  on  each   node  $A$ in this tree, the strategy is  either 

\bi \item 

 in a \emph{betting state}, betting on  smaller and smaller $(p,q)$--sub-intervals of $A$, or 
 
\item  in a \emph{non-betting state}, processing   smaller and smaller   $(r,s)$--sub-intervals of $A$, but  without betting. \ei


The root of the tree is   $A =[0,1]$. Initially let $g(0)= 0$ and $g(1)=1$ (hence     $\Gamma(A )= 1$),  and put the strategy $\Gamma$  into the betting state. 

One technical problem is that we never know a computable real   such as $S_f(A)$ in its entirety; we only have rational approximations. For a real $x$ named  by a  Cauchy sequence as in Subsection~\ref{ss:CompReals}, we let $x_n$ denote the $n$-th term of that sequence. Thus $|x- x_n |\le \tp{-n}$.  To make use of the inequalities in Lemma~\ref{lem:ratspq}, 
we choose  $K\in \NN $   large enough that the inequalities  still hold with $\gamma + \tp{-K}$ instead of $\gamma$, and with $\beta - \tp{-K}$ instead of $\beta$, respectively. We also require that $\beta + \tp{-K}< \gamma - \tp{-K}$.

 Suppose $A=[a,b]$ is an interval such that   $\Gamma(A)$ has already been  defined. By hypothesis $S_f(A)$ is a computable real uniformly in $A$. Proceed according to the case that applies. 
 
 \vsp 
\n  (I): {\it  $\Gamma$ is in the betting state on $A$.}

\vsps 

\n (I.a)  $S_f(A)_K \le  \gamma$.  If $\Gamma$ has  just entered the betting state on  $A$, let \bc  $A = \bigsqcup_k A_k$ \ec  where the $A_k$ form an effective   sequence of $(p,q)$--intervals that are disjoint (on $[0,1] \setminus \Q$). Otherwise, split  $A = A_0 \cup A_1$ into disjoint intervals of equal length.

 The function $g$    interpolates between $a$ and $b$ with a growth proportional to the growth of   $f$:  if  $v \in (a,b)$ is  an endpoint of a new interval,  define
\[ g(v) = g(a) + (g(b)-g(a)) \frac{f(v)- f(a)}{f(b) - f(a)}.\]
    Continue the strategy on each sub-interval.

\vsps
%


\n   (I.b)  $S_f(A)_K > \gamma$. Switch  to the non-betting state on   $A$ and goto (II).  

\vsp

\n  (II)   {\it  $\Gamma$ is in the non-betting state on $A$.}     
\vsps

\n   (II.a)  $S_f(A)_K \ge   \beta$. If $\Gamma$ has just entered the non-betting state on $A$, let  \bc $A =\bigsqcup_k A_k$ \ec  where the $A_k$ form an effective   sequence of $(r,s)$--intervals that are disjoint on $[0,1] \setminus \Q$, and further,  $2|A_k| \le |A|$ for each $k$. Otherwise split  $A = A_0 \cup A_1$ into disjoint intervals of equal length.

If  $v \in (a,b)$ is  an endpoint of a new interval, then $g$   interpolates linearly:  let 
\[ g(v) = g(a) + (g(b)-g(a)) \frac{v- a}{b - a}.\]
Continue the strategy  on each sub-interval.


\vsps

\n   (II.b)  $S_f(A)_K < \beta$.  Switch  to the  betting state on   $A$ and goto (I).


\subsubsection{The verification.}
If the strategy  $\Gamma$,  processing an interval $A = [a,b]$ in the betting state,  chooses  a sub-interval  $[c,d]$, then 

\[ g(d)-g(c) = (g(b)-g(a)) \frac{f(d)- f(c)}{f(b) - f(a)}.\]
Dividing this equation by $d-c$ and recalling the definition of  the $\Gamma$-values  in (\ref{eqn:Gamma}),   we obtain

\begin{equation} \label{eqn:webet} \Gamma([c,d]) = \Gamma([a,b])  \frac{S_f(c,d)}{S_f(a,b)}. \end{equation}

The purpose of the following two claims is to extend $g$ to a computable function on $[0,1]$.  	For the rest of the proof,  we will use the shorthand
	 \bc  $g[A] =g(b)-g(a)$ \ec  for  an interval on the tree $A= [a,b]$. 
	 Recall that  we    write $S_g(A)$ for the slope  $S_g(a,b)$. Thus $\Gamma(A) = S_g(A)= g[A]/|A| $. 
\begin{claim} \label{cl:limtreepath} Let $x \in [0,1]$. Let $B_0 \supset B_1 \supset \cdots$ be an   infinite path on the tree of    intervals. Then  $\lim_m g[B_m]=0$. \end{claim}
	
 	 We  consider the states of the betting strategy $\Gamma$  as it processes  the intervals $A= B_m$.

	\vsps

	\n {\it Case 1:}  $\Gamma$ changes its  state  only  finitely often when  processing the intervals~$B_m$.    

	\n If $\Gamma$ is eventually in a non-betting state then clearly $\lim_m g[B_m]=0$. Suppose otherwise, that is, $\Gamma$ is eventually in a betting state. Suppose further that $\Gamma$ enters the betting state for the last time when it defines the interval $A = B_{m^*}$.    Then for all  $m  \ge m^*$,  by (\ref{eqn:webet}) and since  $S_f(B_m)_K \le \gamma$, we have 

	\[\Gamma(B_m )= \Gamma(B_{m^*})   \frac{S_f(B_m)}{S_f(B_{m^*})} \le   (\gamma + \tp{-K})\frac{\Gamma(B_{m^*})}{S_f(B_{m^*})}=:C.\]
	 Hence $g[B_m] =  \Gamma(B_m) \cdot |B_m|   \le C |B_m|  $.  

	\vsps

	\n {\it Case 2:} 
	  $\Gamma$ changes its  state infinitely often when  processing the intervals~$B_m$.  

	 Let $B_{m_i}$ be the interval $A$ processed when the strategy is for the $i$-th time   in a   betting state at (I.b). Note that  $g[B_{m_i+1}]\le g[B_{m_i}]/2$ because at (II.a) we chose all  the splitting components $A_k$ at most half as long as the given interval $A$.  Of course, by monotonicity of $g$  we have $g[B_{m+1} ] \le g[B_m]$ for each $m$. Thus,   $ g[B_{m}] \le \tp{-i} $ for each $m > m_i$. This completes the proof of the claim.

\begin{claim} The function  $g$  can be extended to a computable function on $[0,1]$. \end{claim}

 \n Let $V$ be the set of endpoints of intervals on the tree. Clearly $V$ is dense in $[0,1]$.   For $x\in [0,1]$ let 
\begin{eqnarray*}  \ul g(x) & =&  \sup \{g(v)\colon \, v< x, v \in V \} \\  \ol g(x) & = &  \inf \{g(w)\colon \, w> x, w \in V\}. \end{eqnarray*}
 We show that $\ul g(x) \ge \ol g(x)$. Since $g$ is nondecreasing on $V$, this will imply that $\ul g = \ol g $ is a continuous extension of $g$.

There is an infinite path $B_0 \supset B_1 \supset \cdots$  on the tree of    intervals   such that $x \in \bigcap_m B_m$.
By Claim~\ref{cl:limtreepath}, we have \bc  $\ul g(x) \ge  \sup_m g( \min \, B_m) =  \inf_m g( \max \, B_m)  \ge \ol g(x)$. \ec

 Clearly there  is a computable dense sequence of rationals  
$\{v_i\}\sN i$    that lists without repetition   the set $V$ of  endpoints of intervals in the tree. By definition, $g(v_i)$ is a computable real uniformly in $i$.   
 Since $\ul g$ is continuous nondecreasing,  by Proposition~\ref{prop:monotonic computable} we may conclude that $\ul g$ is computable.  This establishes the claim.

From now on we will use the letter  $g$ to denote  the function extended to $[0,1]$.


\begin{claim}   \label{claim:Dgzinf} We have  $\ol Dg(z) = \infty$. \end{claim}

\n Let $\CCC$ denote  the tree  of intervals built  during the construction. Note that for each $\epsilon>0$ there are only finitely many intervals in $\CCC$ of length greater than $\epsilon$.   To prove the   claim, we show that the strategy $\Gamma$ succeeds on $z$ in the sense that \bc $\sup_{z \in A \in \CCC} \Gamma(A)= \infty$. \ec By the definition of the $\Gamma$-values in (\ref{eqn:Gamma}) this will  imply  $\ol Dg(z) = \infty$: let $([a_n, b_n])\sN n$ be  a sequence of intervals containing $z$  such that $\Gamma([a_n, b_n])= S_g(a_n, b_n)$ is unbounded and  $ \lim_n (b_n - a_n) = 0$. If $z\in V$,  then necessarily $a_n = z$ or $b_n=z$ for almost all $n$. This  clearly implies $\ol Dg(z) = \infty$. If $z \not \in V$, then  $a_n, b_n \neq z$ for all $n$, and we have  
$S_g(a_n, b_n) \le  \max \{  S_g(a_n, z),S_g(z,b_n) \}$ by Fact~\ref{fact:slopes}.   This  also implies that  $\ol Dg(z) = \infty$. 

Now we come to the crucial argument why   $\Gamma$ succeeds, first  we verify that $\Gamma$ changes its state infinitely often on intervals $B $ such that $z \in B$. Suppose $\Gamma$ entered the betting state in (II.b) and hence jumps to (I.a).   Following the notation   in (I.a), let $A_k$ be the $(p,q)$--interval  containing $z$. By the first line in Lemma~\ref{lem:ratspq} and the definition of $K$ in~\ref{sss:Defg},  there is a $(p,q)$--interval $A \sub A_k$ containing $z$ such that $S_f(A) > \gamma+\tp{-K}$. Thus $S_f(A)_K > \gamma$ and $\Gamma$   enters the non-betting state when it processes this interval, if not before.

Similarly, once $\Gamma$ enters the non-betting state on an interval $A_k$ containing~$z$, by the second line of Lemma~\ref{lem:ratspq} it will revert to the betting state on some  $(r,s)$--interval $B\sub A_k$ containing $z$.

Now suppose   $\Gamma$ enters the betting state on $A$, $B$ is a largest sub-interval of $A$ such that $\Gamma$ enters the non-betting state on $B$, and then again,  $C$ is a largest  sub-interval of $B$ such that $\Gamma$ enters the  betting state on $C$. Then $S_f(A)_K < \beta $ while $S_f(B)_K > \gamma$,  so $\Gamma(B) = \Gamma(A) S_f(B)/S_f(A) \ge    \Gamma(A)\delta$ with $\delta=\frac{\gamma-\tp{-K}}{\beta+\tp{-K}} >1$. Also $\Gamma(B) = \Gamma(C)$. Thus, after the strategy has entered the betting state for $n+1$  times on intervals containing $z$, we have $\Gamma(A) \ge \delta^n$.  This implies that $\Gamma$ succeeds on~$z$.


\begin{remark} \label{rm:scaling}  Suppose we are given a computable function~$f$ as in Theorem~\ref{thm:CRdiff} by an  index in the sense of Subsection~\ref{ss:compfunctions}. The method of the  foregoing proof  enables us to uniformly obtain an index for a  computable nondecreasing  function $p$ such that $f'(z) \UA $ implies $\ol D p(z) = \infty$ for all $z \in [0,1]$.  We simply sum  up all the possibilities for $g$.
This list of  possibilities is effectively given: we have   $f$ itself  (for the case that already $\ul Df(z) = \infty$),  and   all the   functions $g$ obtained  for  any possible values of the rationals $p,q,r,s$ and $0 \le \beta< \gamma$ in  the construction above. 
\end{remark}

\subsection{Proof of (i)$\to$(iii)} 
 Suppose $\ol Dg (z) = \infty$ where    $g \colon \, [0,1] \ria \mathbb R$  is a computable nondecreasing  function.   We may assume that $z$ is irrational. 
We want to show that $z$ is not computably random.  We apply  Lemma~\ref{lem:rat_intervals}    for some fixed $\aaa>1$,  obtaining a finite set $L \sub [-1,1]$ of rationals. There are      $p,q$ in $L$,  $p>0 $,    such that   
 \begin{equation} \label{eqn:bigsup}  \infty = \sup \{S_g(A) \colon \, A \ttext{is a} (p,q)\text{--interval} \lland z \in A\}.\end{equation}
 For a binary string  $\sss$, recall that $[\sss]$ is the closed basic dyadic interval determined by $\sss$.  Let   \bc  $A_\sss = p [\sss]+q$.  \ec    We may assume that  the  given computable nondecreasing  function  $g$ is  actually defined on $[-1, 2]$, so that $S_g(A_\sss)$ is defined for each $\sss$. To do so we let $g(x) = g(0)$ for $x \in [-1,0]$ and $g(x) = g(1) $ for $x \in [1,2]$ and note that this extended function is computable by Proposition~\ref{prop:monotonic computable}.  We define a computable martingale $N$ by 
 \[ N(\sss) = S_g(A_\sss). \]
 Now let $w $ be the irrational number $(z-q)/p$. Then $N$ succeeds on the binary expansion of  the fractional part $w- \lfloor w \rfloor$. For,  given $c>0$, by (\ref{eqn:bigsup})  let  $\sss$ be a string such that $z \in A_\sss$ and $S_g(A_\sss) \ge c$. Then $\sss $ is an initial segment of the binary expansion of $w- \lfloor w \rfloor$ and $N(\sss) \ge c$. 
 
 It follows from the base invariance of computable randomness proved in Theorem~\ref{thm:CRbaseinv} that $z = wp + q$ is also not computably random.
   \end{proof}

\subsection{Proof of Lemma~\ref{lem:rat_intervals}}
 \label{ss:Lemmaproof}  
 We may assume that $0 < x< y < 1/2$.
 Let $k$ be an odd prime number such that $1+8/k < \alpha$. Let 
 
 \begin{eqnarray*} P & = &  \{l/k \colon \, l\in \NN \lland k/2 <  l \le k \} , \\
 				Q & = &  \{v/k \colon \, v\in \ZZ \lland |v| \le k \} \end{eqnarray*}
				
 We claim that 
\bc $L = P \cup PQ$\ec
is a finite set of rationals  as required.  Informally  speaking,  $P$ is a set of scaling  factors for intervals,  and $PQ$ is a set of  possible shifts for intervals. 

\vsps

\n \emph{Finding $A$.}  To obtain $A\supset [x,y]$, let $n \in \NN$ be largest such that $y-x < (1-1/k)\tp{-n}$, and let  $\eta = 1/(2^n k)$. Informally  $\eta$ is the ``resolution'' for a discrete version of the picture that will suffice to find $A$ and $B$.  By the definitions we have 
\begin{equation} \label{eqn:a} y-x + \eta < \tp{-n}. \end{equation}
Pick the  least scaling  factor $p\in P$ such that 
\begin{equation} \label{eqn:b} y-x + \eta  < p \tp{-n}. \end{equation}
Note that  $p > \min P$: if $p = \frac{k+1}{2k}$ then $y-x + \eta < p \tp{-n}$ implies $y-x < (1-   1/k) \tp{-n-1}$ contrary to the maximality  of $n$.  Therefore  we have
\begin{equation} \label{eqn:c} p \tp{-n}  \le  y-x +2\eta. \end{equation}
%
Let $M\in \NN$ be greatest such that $M\eta < x/p  $. Now comes the key step: since $k$ and $2^n$ are coprime, in the abelian group $\QQ /  \ZZ$, the elements $1/k$ and $1/2^n$ together generate the same cyclic group as $ \eta $. Working still in $\QQ /  \ZZ$, there are $i, v_0 \in \NN$,   $0\le i < 2^n$,  $v_0 \le k$
  such that $[i/2^n] + [v_0/k] = 
[M\eta]$.  Then, since $M \eta \le 1$,  there is an integer $v$,   $|v| \le k$,  such that   

\begin{equation} \label{eqn:d} i/2^n + v/k = M \eta. \end{equation}
To define the $L$--interval  $A$, let $q = v/k \in Q$. Let  

\bc $A = p [i\tp{-n}, (i+1)\tp{-n}] + pq$. \ec 
Write $A = [a,b]$. We verify that $A$ is as required.

\n Firstly, $a = pi\tp{-n} +pq = pM \eta < x$, and $x-a \le p  \eta  \le \eta$ because of the maximality of $M$ and because $p \le 1$.

\n Secondly, $|A| = p\tp{-n}$, so we have by (\ref{eqn:b}) and (\ref{eqn:c}) that $y < b < y +2\eta$.  Then

\bc $|A| \le  y-x + 2 \eta = y-x + 2/(2^n k) \le y-x + 8(y-x)/k$, \ec
where the last inequality holds because $\tp{-n} \le 4(y-x)$ by the maximality of $n$. Thus $|A|/(y-x) \le 1+8/k < \alpha$, as required.

\vsps

\n \emph{Finding $B$.} Let $\alpha = 1+ 2 \eps$. The second statement of the lemma can be derived from the first statement for the  precision factor $1+\eps$.  Let $L$ be the finite set of rationals obtained in the first statement for $1+\eps$ in place of $\alpha$.  Given an interval  $[x,y]$, let $[u,v] \sub [x,y]$ be the sub-interval   such that 

\bc $ u-x = y-v = \eps (v-u)$. \ec 

By the first statement  of the lemma there is an $L$--interval $B= [a,b] \supseteq  [u,v]$ such that  $|B|/(v-u) <  1+\eps$. Then 
\begin{eqnarray*}  u-a  & <  & \eps(v-u) = u-x \\
				b-v & <   & \eps (v-u) = y-v,  \end{eqnarray*}
whence $B \subset  [x,y]$. Clearly, $ (y-x)/|B| <  (y-x)/(v-u) = \alpha$.

\begin{remark}  \label{rmk:rats example} To illustrate the lemma and its proof,  suppose $\aaa =4$. We can choose $k$ to be the prime $3$. This yields  a   set $L$ of  at most 16 rationals.  We have  $ P=  \{  \frac 2 3, 1\}$, but the proof shows that in order to find $A$ we never choose $p = \min P$. Thus $p=1$. The shift parameter $q$ is of  the form $v/3$, where $v$ is an integer and $|v| \le 3$. Thus, every interval $[x,y]$ is contained in a basic dyadic interval shifted by some $q$, and of length less than $4 (y-x)$. A  similar fact can be shown with the usual ``1/3-trick'': the endpoints of a basic dyadic interval of length $\tp{-m}$, and another basic dyadic interval shifted by $1/3$ and of the same length, are at least $\tp{-m}/3$ apart.
\end{remark}

\section{Consequences of Theorem~\ref{thm:CRdiff}}
\label{ss:Pathak}
In this section we       provide some interesting  consequences of Theorem~\ref{thm:CRdiff} and its proof. We say that a  real $z\in \RR$ satisfies an algorithmic  randomness notion if its fractional part  $z - \lfloor z \rfloor$  satisfies it. 


\begin{cor} Each computable nondecreasing function  $f$ is differentiable at a computable real. Moreover, the real can be obtained uniformly from an index for $f$.  \end{cor}

\begin{proof}  By  Remark~\ref{rm:scaling}  above, from an index for $f $ we can compute an index  for a nondecreasing function $g$ such that $f'(z) \UA $ implies $\ol D g(z) = \infty$ for all $z \in [0,1]$.   Our first goal is to show that  one can   compute an index for a function $h$ such that   $\ul Dh(z) = \infty$ in case $f'(z) \UA $, for each  $z \in [1/3, 2/3)$. The idea is to turn appropriate  martingales into martingales with the savings property, and then apply   the implication (i)$\to$(ii) of  Theorem~\ref{thm:MonSuc}. 

  We use  the simple case of Lemma~\ref{lem:rat_intervals} with the parameters as in the foregoing Remark~\ref{rmk:rats example}. Let $q$ range over rationals of  the form $v/3$, where $v$ is an integer and $|v| \le 3$.  

Let     $ N_{q}$ be   the computable  martingale $N$ obtained in the proof of (i)$\to$
(iii) of Thm.\ \ref{thm:CRdiff} 	above for   $p=1$.   
 By Proposition~\ref{prop:Savings} from an index for  $N_q$ we can compute a  martingale $M_{q}$ with the
 savings property that succeeds on the binary expansion of a real $u \in [0,1]$ if $N$ does.

 Let $x$ range over  $ [1/3, 2/3]$. For $x < 2/3$  let $v_q(x)$ be the fractional part of $x-q$;  let $v_q(2/3) = \lim_{x < 2/3, x \to 2/3} v_q(x)$.    Let $h_q(x) =  \cdf(M_q) (v_q(x))$, and let $h(x) = \sum_q h_q(x)$.    Clearly $h$ is computable from an index for $f$.

Now consider  $z\in [1/3,2/3)$  such that $f'(z)\UA$. Then  $  \ol Dg(z)= \infty$, so for some $q$, $N_q$ succeeds on the binary expansion of $v_q(z)$ as observed in the proof of (i)$\to$(iii) of Thm.\ \ref{thm:CRdiff}.    Hence   $M_q$  succeeds on the binary expansion of $v_q(z)$, which by the   implication (i)$\to$(ii) of  Theorem~\ref{thm:MonSuc} implies that $\ul D h_q(z) = \infty$. Therefore $\ul Dh(z) = \infty$. 

Let $r(x) = h(x) - h(1/3)$, extend this to a computable nondecreasing function on $[0,1]$ by assigning the  value $0$ to $y<1/3$, and the value  $h(2/3)$ to $y>2/3$, and  let $V $ be the computable martingale $\Mart(r)$.  We have $\cdf(V) = r$ by Fact~\ref{fa:correspondence}. By the implication (ii)$\to$(i) of  Theorem~\ref{thm:MonSuc} (which does not rely on the savings property), we may conclude that $V$ succeeds on the binary expansion of $z$.
Note that an index for $V$, viewed as a function from binary strings to Cauchy names for reals, is computable  from an index for $g$, and hence from an index  for $f$.

	It remains to compute, from an index for  $V$, the binary expansion $Z$ of  a real  $z\in [1/3,2/3)$ such that $V(Z\uhr n)$ is bounded. Let the first 3 bits of $Z$ be $1,0,0$. For $n\ge 3$, if $\sss = Z\uhr n$ has been determined, use $V$ to determine a bit $Z(n)= b$ such that $V(\sss \ape b)\le V(\sss \ape  (1-b))+\tp{-n}$. Clearly, $\sup_n V(Z\uhr n)< \infty$. 
 \end{proof} 
%
Note that in the argument  above, different indices for $f$ might result in different reals.
Next  we obtain a preservation result for computable randomness.  For instance,  computable  randomness is preserved under the map    $z \mapsto e^z$, and,   for each computable  real $\aaa \neq  0 $, under the map $z \mapsto  z^\aaa$.
\begin{cor} 
\label{cor:diff-bij}  Suppose $z \in \RR$ is computably random. Let $H$ be a computable function that is   1-1 in a neighborhood of $z$. If~$H'(z) \neq 0$, then $H(z)$ is computably random.
 \end{cor}
 
\begin{proof} Note that $H'(z)$ exists by Theorem~\ref{thm:CRdiff}.  First suppose $H$ is increasing  in a neighborhood of $z$.  If a function $f$   is computable and nondecreasing in a neighborhood of $H(z)$, then 
the composition $f\circ H$  is nondecreasing  in a neighborhood of $z$. Thus, since $z$ is computably random, $(f\circ H)'(z)$ exists. Since $H'(z) \neq  0$,
this implies that $f'(H(z))$ exists. Hence $H(z)$ is computably random. 

If $H$ is decreasing, we  apply the foregoing argument to $-H$ instead.
\end{proof}
 
 \begin{cor} \label{cor:Lipschitz} If a real $z\in [0,1]$ is computably random, then each computable Lipschitz function $h$  on the unit interval is differentiable at $z$. 
	\end{cor}
	\begin{proof} Suppose $h$  is Lipschitz via a constant $C \in \NN$. Then the function~$f$ given by $f(x)= Cx- h(x)$ is  computable and  nondecreasing. 
	Thus, 	 by (i)$\to$(ii) of  Theorem~\ref{thm:CRdiff}, $f$ and hence $h$ is differentiable at $z$.
	\end{proof}
The converse of Cor.\ \ref{cor:Lipschitz} has been shown in \cite{Freer.Kjos.ea:nd}.  Thus, monotonicity can be replaced by being Lipschitz in Theorem~\ref{thm:CRdiff}.
Note that the function~$f$ in the foregoing proof is  Lipschitz. Thus, computable randomness is characterized by differentiability of computable functions that are monotonic, or Lipschitz, or both monotonic and Lipschitz.  This accounts for some arrows in  Figure~\ref{fig:diagram} in the introduction.

%
%
%
%

\subsection{The Denjoy alternative for computable functions}   \label{ss:DAlt}

Several theorems in classical  real analysis say that a certain function is well-behaved  almost everywhere. Being well-behaved can mean other things than being differentiable, although it is usually closely related. In the next two subsections we give two examples of such results and discuss their effective versions.    As before, $\leb$ denotes the usual Lebesgue measure on the unit interval.

The first result applies to arbitrary functions on the unit interval.   For  a somewhat  more general Theorem, see  \cite[Thm. 5.8.12]{Bogachev.vol1:07} or~\cite[7:9.5]{Bruckner:78}.

\begin{thm}[Denjoy, Saks, Young] \label{thm:Denjoy} Let $f$ be a   function  $[0,1]\ria \R$. Then   $\leb$-almost surely, the \emph{Denjoy alternative} holds at $z$:

\bc either $f'(z)$ exists, or   $\ol Df(z) =  \infty$ and $ \ul Df(z) = -\infty$.   \ec
\end{thm}
 

The Denjoy alternative for effective functions was first studied by Demuth (see \cite{Kucera.Nies:12}). The following result is a combination of work  by Demuth,  Miller, Nies, and \Kuc{}. In contrast to  Theorem~\ref{thm:CRdiff}, it characterizes computable randomness in terms of  a differentiability property of computable functions, without any additional conditions on the function.
\begin{theorem}  \label{thm:DenjoyCR}  Let  $z \in [0,1]$. Then
   $z$ is computably random $\LR$ 
   
\hfill for every computable $f \colon \, [0,1] \to \mathbb R$ the Denjoy alternative holds at~$z$.      \end{theorem}

\begin{proof} For the implication ``$\LA$'', let  $f$ be a nondecreasing computable function. Since  $f$ satisfies the Denjoy alternative at $z$ and  $\ul Df(z) \ge 0$, this means that $f'(z)$ exists.
Thus   $z$ is computably random by  Theorem~\ref{thm:CRdiff}.

For a proof of  the implication ``$\RA$'' see \cite{Bienvenu.Hoelzl.ea:12}.
\end{proof} 

\subsection{The Lebesgue differentiation theorem and  $\+ L_1$-computable functions }

  \label{ss:ClassAn}

	The following   result is usually called the  Lebesgue differentiation theorem. For a proof,  see  \cite[Section 5.4]{Bogachev.vol1:07}.
	\begin{thm} \label{thm:LDiffT}  Let $g \colon \, [0,1]\ria \R$  be integrable. Then   $\leb$-almost surely, 

	\[g(z) =  \lim_{r \ria 0}\frac 1 {r}  \int_{z}^{z+r} g  \,  d\leb. \]
	\end{thm}

Before we discuss an effective version of this, we need some basics on $\+ L_1$-computable functions in the sense of  \cite[p.\ 84]{Pour-El.Richards:89}.  Recall that $\+ L_1([a,b])$   denotes the set  of integrable functions $g: [a,b]\ria \R$, and   $||g||_1 = \int_{[a,b]} |g| d\leb$. We say that  $g: [0,1]\ria \R$ is \emph{$\+ L_1$-computable} if there is a uniformly computable sequence $(h_n)\sN n$ of    functions  on $[0,1]$ such that $||g- h_n ||_1 \le \tp{-n}$ for each~$n$.  (The notion of computability for the $h_n$ is the usual one in the sense of Subsection~\ref{ss:compfunctions}; in particular, they are continuous.)

For a function $h$, we  let $h^+ = \max (h,0) $ and $h^- = \max (-h, 0)$. 
 If $g$ is $\+ L_1$-computable via $(h_n)\sN n$, then  $g^+$ is $\+L_1$-computable via $(h_n^+)\sN n$, and $g^-$ is $\+L_1$-computable via $(h_n^-)\sN n$. (This follows because $|g^+(x)- h_n^+(x)| \le |g(x) - h_n(x)|$, etc.)

If  $g \in \+L_1([0,1])$ then its restriction  to the interval $[0,x]$ is in $ \+L_1([0,x])$.  Let  
 $G(x) = \int_0^x g d\leb$.  We have $G= G^+-G^-$ where   $G^+(x) = \int_0^x g^+  d\leb$ and $G^-(x) = \int_0^x g ^-d\leb$. 
\begin{fact} \label{fa:parts} If $g$ is $\+L_1$-computable then $G^+$ and $G^-$ are computable. \end{fact}
\begin{proof}   For each $\+L_1$-computable function  $f$, $\int_0^q f d\leb$ is computable uniformly in a rational $q$ by \cite[Lemma 2.3]{Pathak:09}. Thus the nondecreasing continuous functions $G^+, G^-$ are computable  by Proposition~\ref{prop:monotonic computable}. \end{proof}

By Theorem~\ref{thm:LDiffT},  if  $g$ is in $\+L_1([0,1])$ and $G$ is as above,  then   for $\leb$-almost every~$z$, $G'(z)$ exists and equals $g(z)$. 
For  the mere existence of $G'(z)$, we have the following.

\begin{cor} \label{cor:L1-comp}  Let $g$ be $\+L_1$-computable. Then $G'(z)$ exists for each computably random real $z$. \end{cor}
\begin{proof} It suffices to note that by  Fact~\ref{fa:parts} and Theorem~\ref{thm:CRdiff},  $(G^+)'(z)$ and $(G^-)'(z)$ exist.
\end{proof}

The condition in Cor.~\ref{cor:L1-comp} that $G'(z)$ exists for each $\+ L_1$-computable function~$g$  actually  characterizes Schnorr randomness by recent  results of Rute~\cite{Rute:12}, and Pathak, Rojas and Simpson already  mentioned in the introduction.

%
%
%
%
%
%
%

\section{Weak $2$-randomness and \ML{} randomness}
\label{s:Weak2_and_ML}


 In this section, when discussing inclusion, disjointness, etc., for open sets in the unit interval, we will  ignore    the elements that are  dyadic rationals.  For instance, we view the interval $(1/4, 3/4)$ as the union of $(1/4, 1/2)$ and $(1/2, 3/4)$. With this convention,   the clopen sets in Cantor space $\cantor$ correspond to the finite unions of open  intervals with dyadic rational endpoints.

\subsection{Characterizing weak 2-randomness in terms of differentiability}

 Recall that a real $z$ is \emph{weakly $2$-random} if $z$ is in no null $\PI 2$ set.

\begin{theorem} \label{thm:w2rChar} Let  $z\in [0,1]$. Then

	\n $z$ is weakly $2$-random $\LR$ 
	
	\hfill each a.e.\ differentiable computable function is differentiable at $z$. \end{theorem}
	
	\begin{proof}

\n ``$\RA$'':  For   rationals $p,q$ let 
\begin{eqnarray*}  \ul C(p) & = & \{z \colon \,  \fa t >0 \, \ex h  \, \big [ 0< |h| \le t  \lland  \, S_f(z, z+h) < p \big ] \} \\
\ol C(q) & = & \{z \colon \,   \fa t >0 \, \ex h  \, \big [ 0< |h| \le t  \lland  \, S_f(z, z+h) > q \big ] \},
\end{eqnarray*}
	where $t,h$ range over   rationals. The function  $z \mapsto S_f(z,z+h)$ is computable, and its index in the sense of Subsection~\ref{ss:compfunctions} can be obtained uniformly in $h$. Hence  the set  \bc $\{z \colon S_f(z, z+h) <  p\}$ \ec  is a $\SI 1$ set  uniformly in $p,h$ by Lemma~\ref{lem:comp_function_SI1O} and its uniformity in the strong form remarked after its  proof. Thus $\ul C(p)$ is a $\PI 2$ set  uniformly in $p$. Similarly,  $\ol C(q)$ is a $\PI 2$ set uniformly in $q$. Clearly, 

\vsps

\bc \begin{tabular}{lllll}	
	$\ul Df(z) <p $ & $ \RA$ & $ z \in \ul C(p) $ & $ \RA $ & $\ul Df(z) \le p$,  \\ \\ 
	
		$\ol Df(z) > q$ & $  \RA$ & $ z \in \ol C(q) $& $ \RA $ & $ \ol Df(z) \ge q$. 
	\end{tabular} \ec

\vsps

	Therefore  $f'(z)$ fails to exist iff 
	
	$$\fa p \, [ \, z \in \ul C(p)] \llor \fa q \, [ \, z \in \ol C(q)] \llor \ex p \ex q [p< q \lland z \in \ul C(p) \lland z \in \ol C(q)], $$ 

\item 
\item 	
\n 	where $p, q$ range over rationals. This shows that $\{z \colon \, f'(z) \ \text{fails to exist} \}$ is a $\SI 3 $   set (i.e., an effective union of $\PI 2$ sets).  If  $f$ is a.e.\ differentiable then this set is null and thus cannot contain a weakly $2$-random.
	
	\vsps
	
\n ``$\LA$'': 
For an interval $A\sub [0,1]$ and $p \in \NN $  let $\Lambda_{A, p}$ be the ``sawtooth function'' that is constant 0 outside $A$, reaches $p |A|/2 $ at the middle point of $A$ and is linearly interpolated elsewhere.   Thus  $\Lambda_{A, p}$  has slope $\pm p$ between pairs of  points in the same half of  $A$, and 

\begin{equation} \label{eqn:Lambdabound}\Lambda_{A, p } (x) \le p |A|/2, \end{equation} 
for each $x$.

	 Let $(\mathcal G_m)\sN m$ be a   sequence of uniformly $\SI 1$ sets in the sense of Subsection~\ref{ss:arithmetical_complexity}, where $\mathcal G_m \sub [0,1]$, such that  $\mathcal G_m \supseteq \mathcal G_{m+1}$ for each $m$. We build a computable function $f$ such that $f'(z)$ fails to exist for every  $z \in \bigcap_m \mathcal G_m$.  To  establish the implication $\LA$, we also  show in Claim~\ref{claim:a.e.diff}   that  the function $f$ is  a.e.\ differentiable in    the case that $\bigcap_m \mathcal G_m$ is null.

 Recall the convention that we ignore    the dyadic rationals when discussing inclusion, union, disjointness, etc.,  for open sets in the unit interval. We have  an  effective enumeration    $(D_{m,l})\sN {m,l}$ of  open intervals with dyadic  rational endpoints  such that \bc $\mathcal  G_m= \bigsqcup_{l\in \NN} D_{m,l},$ \ec for each $m$ (the symbol $\bigsqcup$ indicates a disjoint union). We may assume, without loss of generality, that for each $m,k$, there is an $l$ such that $D_{m+1,k}\subseteq D_{m,l}$.
 
We    construct  by recursion on~$m$     a computable  double sequence $(C_{m,i})\sN {m, i} $  of   open intervals with dyadic  rational endpoints such that  $\bigsqcup_{i} C_{m,i} = \mathcal G_m$,   
 \begin{equation} \label{eqnCprop} C_{m,i} \cap C_{m,k} = \ES \  \text{and} \ |C_{m,i}| \ge |C_{m,k}| \ \text{for} \ i  <  k,  \end{equation}   
and, 	 furthermore, if $B = C_{m,i}$ for $m > 0$,  then there is an interval $A = C_{m-1,k}$ such that 
\begin{equation} \label{eqn:AB} B \sub A \lland |B| \le 8^{-m} |A|.\end{equation}
Each interval of the form $D_{m,k}$ will be a finite union of intervals of the form~$C_{m,i}$.

\vsps

\n {\it Construction of  the  double sequence $(C_{m,i})\sN {m, i} $.} 

\n Suppose $m=0 $,  or $m>0$ and  we have already defined  $(C_{m-1,j})\sN j $.  Define $(C_{m,i})\sN { i} $ as follows.

\vsps

Suppose $N \in \NN $ is greatest such that  we   have already defined $C_{m,i}$ for $i < N$. When a new interval $D=D_{m,l}$ with dyadic rational endpoints  is  enumerated into $\mathcal G_m$, if $m>0$ we wait until $D$ is contained in a union of intervals $\bigcup_{r\in F} C_{m-1, r}$, where $F$ is finite. This is possible because $D$ is contained in a single interval in $\mathcal G_{m-1}$, and this single interval was handled in the previous stage of the recursion. If $m>0$, let $\delta$ be the minimum of $|D|$ and  the lengths of these finitely many intervals; if $m=0$,  let $\delta =|D|$.   Let $\eps$ be the minimum of     $|C_{m,N-1}|$ (if $N>0$), and $8^{-m}2^{-l} \delta$. (We will need the factor  $2^{-l}$ when we show  in Claim~\ref{claim:a.e.diff} that~$f$ is a.e.\ differentiable.)

 We  partition $D$ into  disjoint sub-intervals $C_{m,i}$ with dyadic rational endpoints, $i = N, \ldots, N'-1$,    and of nonincreasing length at most $\eps$, so that in case $m>0$ each of the sub-intervals is contained in an interval $A$ of the form $C_{m-1, r}$ for some  $r \in F$.  
For $m \in \NN$ let

\begin{equation*} f_m  =  \sum_{i=0}^\infty \Lambda_{C_{m,i},  4^m},
                       \end{equation*}
                       and let $f =  \sum_{m=0}^\infty f_m$.
Since $|C_{m,i}| \le 8^{-m}$ for each $i$, we have $f_m(x) \le 8^{-m}4^{m}/2 \le  2^{-m-1}$ for each~$x$. 

\begin{claim} The function $f$ is computable. \end{claim} 
Since  $f_m(x) \le 2^{-m-1}$ for each $m$,  $f(x) $ is defined for each $x \in [0,1]$. We  first show that $f(q)$ is computable uniformly in a rational $q$. Given $m>0$, since $|C_{m,i}| \to_i 0$,  we can   find $i^*$ such that \bc $|C_{k,i^*}| \le 8^{-m}/(m+1)$ for each $k \le m$. \ec
Then, since the length of the  intervals $C_{k,i}$ is nonincreasing in $i$ and by (\ref{eqn:Lambdabound}), we have  $\Lambda_{C_{k,i},  4^k}(q)\le 2^{-m-1}/(m+1)$ for all $k\le m$ and $i\ge i^*$.
So by the disjointness in (\ref{eqnCprop}),    $\sum_{k\le m}\sum_{i\ge i^*}\Lambda_{C_{k,i},4^k}(q)\le 2^{-m-1}$.  We also have  $\sum_{k>m}f_k(q)\le \sum_{k>m}2^{-k-1} = 2^{-m-1}$.  Hence 
the  approximation to $f(q)$ at stage $i^*$ based only on the intervals of the form $C_{k,i}$ for $k\leq m$ and $i < i^*$   is within $\tp{-m}$ of $f(q)$.

 To show $f$ is computable, by Subsection~\ref{ss:compfunctions} it suffices now to verify that~$f$ is effectively uniformly continuous.   Suppose $|x-y|\le 8^{-m}$. For $k <  m$, we have  $|f_k(x)-f_k(y)| \le 4^k |x-y|$. For $k \ge  m $  we have  $f_k(x), f_k(y) \le \tp{-k-1}$. Thus
 
  \[|f(x)-f(y)| \le |x-y|  \sum_{k<  m} 4^{k}  + \sum_{k\ge m} \tp{-k} < \tp{-m+2}.\]   

\begin{claim} \label{cl:nondiff} Suppose $z\in \bigcap \mathcal G_m$. Then $\ol D f(z) = \infty$ or $\ul Df(z) = -\infty$. \end{claim}
For each $m$ there is  an  interval $A_m$ of the form $C_{m,i}$  such that 
$z \in A_m$.  Suppose first that  there are infinitely many $m$ such that $z$ is in the left  half of $A_m$. We show $\ol D f(z) = \infty$. Let $m$ be one such value.
Choose 
\[ h = \pm |A_m |/4 \]
  so that $z+h$ is also in the left half of $A_m$. We show that  the slope 
\bc $S_{f_m}(z, z+h) = 4^m$ \ec 
  does not cancel out with the  slopes, possibly negative, that are  due to other $f_k$. If $k< m$ then we have  $|S_{f_k}(z, z+h)| \le  4^k$. Suppose $k>m$.  Then     by (\ref{eqn:Lambdabound}) and  (\ref{eqn:AB})  we have $f_k(x) \le 4^k 8^{-k} |A_m|/2 = 2^{-k-1}|A_m|$ for $x \in \{ z, z+h\}$ and hence  

\[ |S_{f_k}(z, z+h) | \le \frac{2^{-k}|A_m|}{|h|} = 2^{-k+2}.\]
Therefore, for $m>0$ 

\[ S_f(z, z+h) \ge 4^m  - \sum_{k< m }  4^k  - \sum_{k>m} \tp{-k+2} \ge 4^{m-1}-4. \]
Thus   $\ol D f(z) = \infty$. 

If  there are infinitely many $m$ such that $z$ is in the right  half of $A_m$,  then  $\ul D f(z) = -\infty$ by a similar argument.

\begin{claim}\label{claim:a.e.diff} If $\bigcap_m    {\mathcal G}_m$ is null,  then $f$ is  differentiable almost everywhere. \end{claim}
Let $  \hat D_{m,l}$ be the open interval in $\R$ with the same middle point as $D_{m,l}$ such that $|\hat D_{m,l}| = 3 |D_{m,l}|$.  Let $ \hat  {\mathcal G}_m  =  [0,1] \cap \bigcup_l  \hat  D_{m,l}$. Clearly $\leb  \hat  {\mathcal G}_m \le 3 \leb \mathcal G_m$, so that $\bigcap_m  \hat {\mathcal G}_m$ is null.

We show that $f'(z)$ exists for each $z \not \in \bigcap_m   \hat   {\mathcal G}_m $ that is not a dyadic rational. 
In the following, let  $h, h_0$, etc., range over   rationals. Note that 
\bc $S_f(z,z+h)= \sum_{k=0}^\infty S_{f_k}(z, z+h)$. \ec  
 Let $m$ be the least number  such that $z \not \in  \hat  {\mathcal G}_m$. Since $z$ is not a dyadic rational, we may choose $h_0>0$ such that for  each $k< m$, the  function $f_k$  is linear in the interval $[z-h_0, z+h_0]$. So for   $|h|\le h_0$ the contribution of these $f_k$ to the slope $S_f(z,z+h)$ is constant. It now suffices to show that \bc $\lim_{h \to 0}\sum_{r=m}^\infty |S_{f_r}(z,z+h)| =0$. \ec  Note that $f_r$ is nonnegative and $f_r(z) = 0$ for $r\ge m$. Thus  it suffices,  given $\eps >0$,  to find a positive  $h_1 \le h_0$  such that 
\begin{equation} \label{eqn:sumabsfr} \sum_{r=m}^\infty f_r(z+h) \le \eps |h| \end{equation}
whenever  $|h|\le h_1$.

Roughly, the idea is the following: take $r\ge m$. If $f_r(z+h)\neq 0$ then  $z+h$ is in some $D_{m,l}$. Because $z\not \in \hat    D_{m,l}$, $|h| \ge |D_{m,l}|$.  We make sure that   $f_r(z+h)$ is small compared to $|h|$ by  using that the height of the relevant sawtooth depends on  the length of its base interval $C_{r,v}$ containing $z+h$, and that this length is small compared to $h$.

We now provide  the details on how to find $h_1$ as above. Choose $l^* \in \NN$ such that $2^{-l^*}\leq\eps$. If $C_{m,i}\sub D_{m,l}$ and $l\geq l^*$, we have  
\begin{equation} \label{eqn:CshorterD}|C_{m,i}| \le 8^{-m} \eps |D_{m,l}|. \end{equation}
Let $h_1 = \min \{  |D_{m, l}| \colon \, {l < l^*} \}$. Suppose  $h>0$ and   $|h|\le h_1$. 

Firstly we consider the contribution of $f_m$ to  (\ref{eqn:sumabsfr}).  If $f_m(z+h) >  0$ then $z +h \in C_{m,i}\sub D_{m,l}$ for some (unique) $l,i$. Since $z \not \in  \hat  D_{m,l}$ and $|h|\le h_1$, we have  $|h|\ge |D_{m,l}|$ and $l\ge l^*$. By (\ref{eqn:Lambdabound}),  (\ref{eqn:CshorterD}) and the definition of $f_m$, 
 \bc  $f_m(z+h) \le 4^m |C_{m,i}|/2 \le \tp{-m-1} |D_{m,l}| \eps$. \ec  Thus  $f_m(z+h)  \le \tp{-m-1} \eps |h|$.

Next,  we consider the contribution of $f_r$, $r>m$, to  (\ref{eqn:sumabsfr}).
If $f_r(z+h) > 0$ then $z +h \in C_{r,v}\sub C_{m,i}$ for some $v$. Thus, by construction,  
 \bc $f_r(z+h)\le 4^r |C_{r,v}|/2 \le 4^r 8^{-r} |C_{m,i}|/2 \le \tp{-r-1} |D_{m,l}| \eps\le \tp{-r-1} \eps |h|$. \ec
This establishes  (\ref{eqn:sumabsfr}) and completes the proof. \end{proof}

    \subsection{Characterizing \ML{} randomness in terms of differentiability}

\label{ss:CharML}


 Recall that a function $f \colon \,  [0,1]\ria \R$ is   of \emph{bounded variation} if 

\[ \infty > \sup \sum_{i=1}^n   | f(t_{i+1}) - f(t_i)|, \]
where the sup is taken over all collections $ t_1 < t_2 <  \ldots <   t_n$ in $[0,1]$.
 A stronger condition on $f$  is          absolute continuity: for every $\eps > 0$, there is $\delta > 0$ such that \[ \eps  > \sup \sum_{i=1}^n   | f(b_i) - f(a_i)|, \]
for every  collection $ 0\le a_1 < b_1 \le a_2 < b_2  \le \ldots \le a_n < b_n\le 1 $ such that $\delta > \sum_{i=1}^n  b_i - a_i$. The absolutely continuous functions are precisely the indefinite integrals of functions in $\+L_1 ([0,1])$ (see  \cite[Thm. 5.3.6]{Bogachev.vol1:07}).  Note that it is easy to construct a computable differentiable function that is not of bounded variation. 

We will characterize \ML{} randomness via differentiability of computable functions of bounded variation,  following the scheme  $(*)$ in the introduction.  For the implication $\LA$, an appropriate  single function suffices, because there is a universal \ML{} test. 
 
\begin{lemma}[\cite{Demuth:75}, Example 2]  \label{lem:MLtest_to_function} There is a computable function $f$ of bounded variation  (in fact,  absolutely  continuous) such that $f'(z)$ exists only for \ML{} random reals $z$. \end{lemma}

\begin{proof} Let $(\mathcal G_m)\sN m$ be a universal \ML\ test, where $\mathcal G_m \sub [0,1]$, such that  $\mathcal G_m \supseteq \mathcal G_{m+1}$ for each $m$.  We may   assume that   $\leb \mathcal G_m \le 8^{-m}$. Define a computable function $f$ as in the proof of the implication $\LA$ of Theorem~\ref{thm:w2rChar}. By Claim~\ref{cl:nondiff},  $f'(z)$ fails to exist for any $z \in \bigcap_m \+ G_m$, i.e., for any $z$ that is not \ML{} random. It remains to show the following.

\begin{claim} \label{claim:facbv} $f$ is absolutely continuous, and hence of bounded variation. \end{claim}

\n For an open interval $A\sub [0,1] $, let $\Theta_{A,p}$ be the function that is undefined at  the endpoints and the  middle point of $A$,  has value $p$ on the left  half, value $-p$ on the right half of $A$, and is $0$ outside $A$. Then $\int_0^x \Theta_{A,p}= \Lambda_{A,p}(x)$. 

Let $g_m = \sum_{i} \Theta_{C_{m,i}, 4^m}$. Note that $\leb \+ G_m \le 8^{-m}$ implies that $g_m$ is integrable with   $\int |g_m| \le \tp{-m}$, and hence $\sum_m \int |g_m| \le 2$. Then, by a well-known corollary to the Lebesgue dominated convergence theorem (see for instance \cite[Thm.\ 1.38]{Rudin:87}),  the function 
 $g(y) = \sum_m g_m(y)$ is defined a.e., $g$ is integrable, and $\int_0^x g= \sum_m \int_0^x g_m$. Since $f_m(x) = \int_0^x g_m$, this implies that $f(x) = \int_0^x g$. 
 Thus,  $f$ is absolutely continuous.
\end{proof}

For a function $h$,  let $h^+ = \max (h,0) $ and $h^- = \max (-h, 0)$, so that $h = h^+ - h^-$. 
Since $g^+ = \sum_m g_m^+$ and $g^-= \sum_m g_m^-$,
 by the monotone convergence theorem (see for instance \cite[Thm. 2.8.2]{Bogachev.vol1:07}), we have \bc $\int_0^x g^+ = \sum_m \int_0^x g_m^+$ and $\int_0^x g^- = \sum_m \int_0^x g_m^-$. \ec
Since $\leb \mathcal G_m \le 8^{-m}$, we have $\int_0^x g_m^+ \le \tp{-m}$ and $\int_0^x g_m^- \le \tp{-m}$, so both sums above  are bounded by $2$.  
Hence the   function $g$ is integrable with $\int_0^x g = \int_0^x g^+ - \int_0^x g^- = f(x)$. 

 
We now arrive at the   analytic characterization of \ML{} randomness originally  due to Demuth~\cite{Demuth:75}.  The implication (i)$\ria$(ii) below restates \cite[Thm.\ 3]{Demuth:75} in classical language. 
 
\begin{theorem}  \label{thm:MLbV}
The following are equivalent for $z\in [0,1]$:

\bi

\itone
  $z$ is Martin-L\"of random.
  
\ittwo   Every computable  function  $f$   of bounded variation is differentiable at~$z$.
\itthree  Every computable  function  $f$   that is absolutely continuous  is differentiable at~$z$.
   \ei
 
\end{theorem}
 \begin{proof}  
 The implication  (iii)$\ria$(i) follows from Lemma~\ref{lem:MLtest_to_function}.  The implication (ii)$\ria$(iii) follows because each absolutely continuous function has bounded variation.  The   implication   (i)$\ria$(ii)   will be     postponed to Subsection~\ref{ss:Proof_MLbV}, where we prove it with a weaker hypothesis on the effectivity of $f$.
\end{proof}
We obtain a preservation  result for \ML{} randomness similar to Corollary~\ref{cor:diff-bij}.   
 For instance,  \ML{}  randomness is also preserved under the map    $z \to e^z$, and,   for each computable  real $\aaa \neq   0 $, under the map $z \to  z^\aaa$. 
\begin{cor} \label{cor:invariance_ML} Suppose $z \in \RR$ is \ML{} random. Let $H$ be a computable function that is Lipschitz and 1-1  in a neighborhood of $z$. If $H'(z) \neq 0$, then $H(z)$ is \ML{} random.
 \end{cor}
\begin{proof}  Let $f$ be an arbitrary function that is computable and absolutely continuous in a neighborhood of $H(z)$. Then 
the composition $f\circ H$  is absolutely continuous in a neighborhood of $z$. Thus, since $z$ is \ML{} random, $(f\circ H)'(z)$ exists. Since $H$ is continuous and 1-1 in a neighborhood of~$z$,  $H'(z) \neq  0$  implies that $f'(H(z))$ exists. Hence $H(z)$ is \ML{} random by Theorem~\ref{thm:MLbV}. 
\end{proof}

In fact it suffices to assume that the function $H$ is   Lipschitz in a neighborhood of $z$. This includes the functions $x \mapsto \sqrt x$ and $x \mapsto 1/x$. Thus, for instance, $\sqrt \Om$   and $1/ \Om$ are  \ML{} random.

Using  Lemma~\ref{lem:MLtest_to_function}  we   obtain an example of an  integrable function $g$ that  is not $\+L_1$-computable, even though its indefinite integral is computable. Let $g$ be the function  from the proof of Claim~\ref{claim:facbv}.  If $g$ were $\+L_1$-computable,  then the function~$f$  from Lemma~\ref{lem:MLtest_to_function}  would be differentiable at each computably random real by  Cor~\ref{cor:L1-comp} and because $f(x) = \int_0^x g$.


\section{Extensions of  the results to weaker effectiveness notions}

\label{s:extensions}

So far, we have    proved two  instances of equivalences of type  $(*)$ at the beginning of the paper:   for weak $2$-randomness, and for  computable randomness. We have also stated a result for \ML{} randomness in Theorem~\ref{thm:MLbV} and proved the    implication $\LA$ in $(*)$; the converse implication will be provided in this section.

 We will see that these equivalences  do not rely on the full hypothesis that the functions in the relevant class are computable.   

\vsp 

\n	\emph{Computability on $I_\Q$.} Recall   that  $I_\Q = [0,1] \cap \QQ$.   We say  that a function $f$  is \emph{computable on $I_\Q$} if its  domain contains $I_\Q$, and $f(q)$ is a computable real uniformly in $q \in I_\Q$. A   function $f$ that is computable on $I_\Q$  and has domain $[0,1]$  need not be continuous: for instance, let $f(x) = 0 $ for $x^2 \le 1/2$, and $f(x)= 1  $ for $x^2 > 1/2$. 
In fact,  computability of a function $f$ on $I_\Q$ is so general that it can barely  be considered    a  genuine notion from computable analysis: we merely require that $(f(q))_{q \in I_\Q}$ can be viewed as  a computable family  of reals indexed by the  rationals in $[0,1]$, similar to the  computable sequences of reals defined in  Subsection~\ref{ss:CompReals}. 
	
Nonetheless, 	in this section we will show that this  much weaker effectivity hypothesis   is sufficient for the implications $\RA$ in $(*)$, including the case of \ML{} randomness. Of course, if $f$ is not defined in a whole neighborhood of  a real~$z$, we lose the usual notion of differentiability at~$z$. Instead, we will  consider pseudo-differentiability at $z$, where one only looks at the slopes at smaller and smaller intervals containing $z$ that have rational endpoints. If $f$ is total and continuous (e.g., if $f$ is computable), then pseudo-differentiability coincides with  usual differentiability, as we will see in Fact~\ref{fact:pseudo_same_continuous}.
Thus, the result for \ML{} randomness  also supplies  our   proof of the implication (i)$\ria$(ii) of Theorem~\ref{thm:MLbV}, which we had postponed to this section. 

The implications $\LA$ in previous proofs of results of type $(*)$ always produce a \emph{computable} function $f$ such that $f'(z)$ fails to exist if the real  $z$ is not random in the appropriate sense. Since computability implies being computable on $I_\Q$, we get full equivalences of type $(*)$ where the effectivity notion is computability on $I_\Q$.

The extensions of our results are interesting because  a number of effectivity notions for functions  have been studied in computable analysis that are intermediate between being   computable and   computable on $I_\Q$. Hence we  also obtain equivalences of type $(*)$ for these effectivity notions.   

  An example of such a notion is Markov computability.    Let  $\phi_e$ denote the $e$-th partial computable function   $\NN \ria \Q$. A  real-valued  function $f$ defined on all computable reals in $[0,1]$ is called  \emph{Markov computable} if  there is a computable function $h\colon \, \NN \to \NN$  such that, if $\phi_e$ is a    Cauchy name of $x$, then $\phi_{h(e)}$   is a   Cauchy name of $f(x)$.  See \cite{Brattka.Hertling.ea:08, Weihrauch:00}, which also discuss   other intermediate effectivity notions for functions  such as the slightly weaker Mazur computability, defined by the condition that   computable sequences of reals are mapped to computable sequences of reals.

Recall from Subsection~\ref{ss:compfunctions} that   a Cauchy name   for a real $x$     is a   sequence $L=(q_n)\sN n$ (i.e., a function $L\colon \, \NN \ria \Q$) such that $|q_n - q_{k}| \le \tp{-n}$ for $k \ge n$. 
Let $\phi_e$ denote the $e$-th computable function   $\NN \ria \Q$. A  real-valued  function $f$ defined on all computable reals in $[0,1]$ is called  \emph{Markov computable} if  there is a computable function $h\colon \, \NN \to \NN$  such that, if $\phi_e$ is a    Cauchy name of $x$, then $\phi_{h(e)}$   is a   Cauchy name of $f(x)$. This  notion has been studied  in the Russian school of constructive analysis, e.g., by Ceitin~\cite{Ceitin:62}, and later      by   Demuth \cite{Demuth:88}, who used the term ``constructive function''. 

Clearly,   Markov computability     implies   computability on $I_\Q$. But Markov computability is much stronger. For instance, each  Markov computable function is continuous on the computable reals. In particular, the  $I_\Q$--computable function  $f$ given above is not Markov computable. 

 To understand this continuity, we discuss an apparently stronger notion of effectivity which is in fact equivalent to Markov computability.   
Recall that in Subsection~\ref{ss:compfunctions} we defined      a  function $f\colon [0,1] \ria \R$  to be  computable if there is a Turing functional $\Phi$ that  maps a  Cauchy name  $L$  of $x\in [0,1]$  to  a Cauchy name    of $f(x)$.
Suppose now      $f(x)$ is at least  defined for  all computable reals $x$ in $[0,1]$. Let us  restrict the  definition above to computable reals:    there is   a Turing functional $\Phi$ such that $\Phi^L$ is   total   for   all \emph{computable} Cauchy names~$L$, and $\Phi$ maps every  computable Cauchy name  for a real $x$ to a Cauchy name for $f(x)$. Note that  such a    function is continuous on the computable reals, because of the use principle: to compute $\Phi^L(n)$, the approximation of $f(x)$ at a  distance of  at most $\tp{-n}$, we   only  use finitely many terms of the Cauchy name $L$ of  $x$.
Clearly,  every   such   function is  Markov computable. The converse implication follows from the Kreisel-Lacombe-Shoenfield/Ceitin theorem; see
     Moschovakis \cite[Thm.\ 4.1]{Moschovakis:10}   for a recent account. 
     
Pour-El and Richards    \cite{Pour-El.Richards:89} gave an example of  a  Markov computable function that is not computable. Bienvenu et al.\ \cite{Bienvenu.Hoelzl.ea:12} provided a Markov computable function $f$ that can be extended to a continuous function on $[0,1]$ and fails the Denjoy alternative  of Subsection~\ref{ss:DAlt}   
at some left-c.e.\ \ML{} random real; the existence of such a function had  already been stated by Demuth \cite{Demuth:76}. Note that  $f$ is not computable by Theorem~\ref{thm:DenjoyCR}. 
%


\subsection{Pseudo-differentiability}

Recall the notations $D^V f(x)$ and $D_V(x)$ from Subsection~\ref{ss:diff_prelims}, where  $V \sub \RR$ and the  domain of the  function $f$ contains  $V \cap [0,1]$. 
We will write  $\utilde{D} f(x)$ for $D_\QQ f(x)$, and $\widetilde D f(x) $ for $D^\QQ f(x)$.

\begin{deff} We say that  a function $f$ with domain containing $I_\Q$    is \emph{pseudo-differentiable at} $x$ if  $-\infty < \utilde  D f(x)  =   \widetilde D f(x) < \infty$. \end{deff}

\begin{fact} \label{fact:pseudo_same_continuous} Suppose that $f\colon [0,1] \to \RR$ is continuous. Then \bc $ \utilde  D f(x)  = \ul Df(x)$ and $ \widetilde  D f(x)  = \ol Df(x)$  \ec for each $x$. Thus,   if $f$ is pseudo-differentiable at $x$,  then $f'(x)= \utilde Df(x) = \widetilde  Df(x)$. \end{fact}
		\begin{proof} Fix $h> 0$.  Since the slope $S_f$ is continuous on its domain,   
		
		\medskip  
			
			\n $\inf  \{S_f(a,b)  \colon a, b \in I_\Q   \lland  \, a\le x \le b \lland\, 0 <  b-a\le h\} $ \smallskip 
			
			\hfill $ \le  \, \inf \{S_f(x, x+l) \colon \, |l| \le h\}$,

		\medskip
			
		\n 	which implies that $\utilde Df(x) \le \ul Df(x)$. The converse inequality is always true by  the remarks at the end of Subsection~\ref{ss:diff_prelims}. In a similar way, one shows that  $\widetilde Df(x) = \ol  Df(x)$. \end{proof}


\subsection{Extension of the results to the   setting of computability on $I_\Q$}
\label{ss:Proof_MLbV}

We will prove        the implications $\RA$ in our three results of type $(*)$ for functions that are merely computable on~$I_\Q$. 
Extending the  definition in Subsection~\ref{ss:CharML}, we  say that a   function $f$ with domain contained in $[0,1]$ is of bounded variation if   $\infty > \sup \sum_{i=1}^n   | f(t_{i+1}) - f(t_i)|$ 
where the sup is taken over all collections $ t_1 < t_2 < \ldots <  t_n$ in the domain of $f$.

\begin{thm} \label{thm:extensionIQcomputable} Let $f$ be computable  on $I_\Q$. 
	
	\bi

\item[(I)] If $f$ is nondecreasing on $I_\Q$, then $f$ is pseudo-differentiable at each  computably random real $z$.

	\item[(II)] If $f\upharpoonright {I_\Q}$ is of bounded variation, then $f$ is pseudo-differentiable at each  Martin-L\"of random real $z$. 
	
	 \item[(III)] If $f$ is  pseudo-differentiable at  almost every  $x\in [0,1]$, then $f$ is pseudo-differentiable at each weakly $2$-random real $z$.

\ei
\end{thm}

\begin{proof} (I) We will show that the analogs of the implications (i)$\to$(iii)$\to$(ii) in Theorem~\ref{thm:CRdiff} are valid when  $\ol Dg(z)$ is replaced by $\widetilde Dg(z)$,  and differentiability by pseudo-differentiability.

 As before, the analog of (i)$\to$(iii) is   proved  by contraposition: if $g$ is nondecreasing and computable on $I_\Q$, and $\widetilde Dg(z) = \infty$,  then $z$ is not computably random. Note that (\ref{eqn:bigsup}) in the proof of (i)$\to$(iii)  is still valid under the hypothesis that $\widetilde Dg(z) = \infty$. The martingale $N$ defined there is computable under the present,  weaker hypothesis that $g$ is computable on $I_\Q$. 

 For the analog of implication (iii)$\to$(ii)  in Theorem~\ref{thm:CRdiff}, we are given a function  $f$ that is nondecreasing and computable on $I_\Q$, and not pseudo-differentiable at $z$. We want to build a function  $g$ that is  nondecreasing and computable on $I_\Q$,   such that $\widetilde Dg(z) = \infty$. 

If $\utilde Df(z)= \infty$ we let $g=f$. Now suppose otherwise. We will show that Lemma~\ref{lem:ratspq}  is still valid for appropriate $\beta < \gamma$. Firstly, we adapt Lemma~\ref{lem:mid3}.  For $h> 0$ we let 
 \bc $\+ K_h = \{\la a,b \ra \colon 0<b-a <  h \lland a + (b-a)/4 <  z  <  b - (b-a)/4 \} $ \ec  
and $\+ K_h^* = \+ K_h \cap \QQ \times \QQ$.
\begin{lemma} \label{lem:mid3-rationals}
 Suppose that  \bc $\lim_{h\to 0} \sup \{S_f(u,v)\colon (u,v)\in \+ K_h^*\} =\lim_{h\to 0} \inf \{S_f(u,v)\colon (u,v)\in \+ K_h^*\}$ \ec and is finite. Then $f$ is pseudo-differentiable at $z$.\end{lemma} 
\n To see this,   take $h>0$ and $t<s$ such that $t<S_f(u,v)<s$ for all $(u,v)\in \+ K_h^*$.   We will use  the notation from  the proof of Lemma~\ref{lem:mid3}. In particular,   we consider  an interval $(c,d)$ with rational endpoints containing $z$ such that $d-c < h/3$, and, as before,  define  $N$, $a_i, b_i $ ($0\le i 
\le N+1)$ so that the intervals  $(a_i, b_i)$ and $(a_{i+1}, b_i)$ ($0 \le i \le N$)  contain $z$ in their middle thirds.   

Note that  $\+ K_h $ is open in $\R^2$. Therefore 

\bc $\+ S = \{ \la u_0, v_0, \ldots, u_N, v_N, u_{N+1}  \ra \colon \fa  i\le N \,  [  \la u_i, v_i \ra \in \+K_h \lland \la u_{i+1}, v_i  \ra \in \+ K_h ]\}$ \ec
is an open subset of  $\R^{2N+3}$ containing $\la a_0, b_0, \ldots, b_N, a_{N+1}\ra $. 

\n   If $\Gamma = \la u_0, v_0, \ldots, u_N, v_N, u_{N+1}  \ra$ is in $\+ S \cap \Q^{2N+3}$ then we have inequality (\ref{line1}) in the proof of Lemma~\ref{lem:mid3} with $u_i, v_i$ instead of $a_i, b_i$. Since $\+ S$ is open we can let such  $\Gamma$ tend to $\la a_0, b_0, \ldots, b_N, a_{N+1}\ra $, which  implies the inequality  (\ref{line1}) as stated.  We  may now continue the argument as before in order to show   that   $ S_f(c,d)<5s-4t$. 

 The lower bound $5t-4s < S_f(c,d)$ is proved in a  similar way.  This yields   Lemma~\ref{lem:mid3-rationals}.

%
%

  By our hypothesis that $f$ is not pseudo-differentiable at $z$,  the limit in   the lemma does not exist, so  we can choose $\wt \beta, \wt \gamma$  such that 
 \begin{eqnarray*}  
	\wt \gamma  &   <  &   \lim_{h \ria 0} \sup \{ S_f(x,y)\colon \,  0 
 \le y-x \le h \lland   \la x,y \ra \in \+ K_h^*  \},  \\
\wt \beta  &   > &    \lim_{h \ria 0} \inf \,  \{ S_f(x,y )\colon \,  0 
 \le y-x \le h \lland    \la x,y \ra \in \+ K_h^* \}. \end{eqnarray*}
 Choosing $\alpha < 4/3$ as before,   the proof of Lemma~\ref{lem:ratspq}   now goes through. 
 
 Note that the construction in the proof of  (iii)$\to$(ii)  actually yields a \emph{computable} nondecreasing $g$.  For, to define $g$ we only needed to compute the values of $f$ on the dense set $V\sub I_\Q$; we did not require  $f$ to be continuous.   

We have $\ol Dg(z) = \infty$ as in Claim~\ref{claim:Dgzinf}. Since $g$ is continuous, by Fact~\ref{fact:pseudo_same_continuous} this implies  $\wt Dg(z)= \infty$. 

\medskip

  Before we  prove (II)     we need some notation. Each $x \in [0,1]$ has a Cauchy name  $(q_n)\sN n$     such that  $q_0 = 0$,  $q_1=1/2$, and each  $q_n$ is of the form $i\tp{-n}$ for an integer $i$. Thus, if $n> 0$ then   $q_{n} - q_{n-1} = a \tp{-n}$ for some $a \in \Sigma  = \{-1, 0,1\}$.  In this way a real $x \in [0,1]$ can be represented by an element of $\Sigma^\omega$. 
     We use this to  introduce  names for   functions $h \colon \, I_\Q  \ria [0,1]$.   Let $(v_n)\sN n$ list $I_\Q$ effectively without repetitions. Via the  representation    given above,  we can name   $h$  by   some  sequence $X \in \Sigma^\omega$: we let $X(\la v_r, n \ra)$ be  the $n$-th entry in a name for $h(v_r)$.  
 It is not hard to show that  the names of nondecreasing functions $h \colon \, I_\Q  \ria [0,1]$ form a $ \PPI$ class. 

Let the variable $q$ range over $I_\Q$.   Jordan's Theorem also holds for functions defined on $I_\Q$: if $g$ has bounded variation on $I_\Q$ then $g \upharpoonright  {I_\Q} = f_0-f_1$ for nondecreasing  functions $f_0,f_1$ defined on $I_\Q$. One simply lets $f_0(q)$ be the variation of $g$ restricted to $ {[0,q] \cap I_\Q}$. Then $f_0$ is nondecreasing. One   checks  as in the usual proof of Jordan's theorem (e.g., \cite[Cor 5.2.3]{Bogachev.vol1:07})   that the function $f_1$ given by $f_1(q) = f_0 (q)-g(q)$ is nondecreasing as well.

 We now prove (II).  We may assume that the variation of $f \upharpoonright {I_\QQ} $ is at most~$1$. Let $\mathcal P$ be the nonempty  class of  pairs  $\la   \eta_0,   \eta_1 \ra$ of names for  nondecreasing functions   $  f_0,    f_1 \colon \, I_\Q \ria [0,1]$  such that $f(q) =   f_0(q) -   f_1(q)$ for each   $q \in I_\Q$.   Since   $f$ is computable on $I_\Q$,  $\mathcal P$ is a $\PPI$ class.

By the ``low for $z$ basis theorem''    \cite[Prop.\ 7.4]{Downey.Hirschfeldt.ea:05},  $z$ is \ML{} random, and hence computably random,  relative to    some member $\la   \eta_0,   \eta_1 \ra$ of $\mathcal P$. 
Thus, by relativizing (I) to both $\eta_0$ and $\eta_1$, we see that  $f_i$ is pseudo-differentiable at $z$ for   $i= 0,1$.  This implies that $f$ is pseudo-differentiable at $z$.

By Fact~\ref{fact:pseudo_same_continuous} this also  provides  the implication (i)$\ria$(ii) of Theorem~\ref{thm:MLbV}.

\medskip

\n   (III). We adapt  the proof of Theorem~\ref{thm:w2rChar} to the new setting.   For any rational $p>0$,  let 
	
	\bc $\utilde C(p) = \{z \colon \,  \fa t >0 \, \ex a,b  [ a \le z \le b \lland 0< b-a \le t \lland   \, S_f(a,b ) < p \}$,  \ec where $t,a,b$ range over rationals.
Since $f$ is computable on $I_\Q$, the set  
\bc $\{z \colon \,  \ex a,b  \,  [ a \le z \le b \lland 0< b-a \le t \lland   \, S_f(a,b ) < p\}$ \ec
  is a $\SI 1$ set  uniformly in $t$.	Then $\utilde C(p)$ is $\PI 2$ uniformly in $p$. Furthermore, 
	$\utilde Df(z) <p \RA z \in \utilde C(p) \RA \utilde Df(z) \le p$. 
	   The sets  $\widetilde C(q)$ are   defined analogously, and  	similar observations hold for them. 

 Now, in order to show that the set of reals $z$ at which $f$ fails to be pseudo-differentiable is a $\SI3$ null set,  we may conclude the argument as before with the notations $\utilde C(p), \widetilde C(q)$ in place of $\ul C(p), \ol C(q)$. 
\end{proof}

  \subsection{Future directions}  \label{ss:future_directions} 
We discuss some current research.

\n \emph{Further algorithmic  randomness notions.}  Miyabe \cite{Miyabe:12} has characterized Kurtz randomness via an effective version of the differentiation theorem. 
  $2$-randomness  has not yet been characterized via  differentiability of effective functions.    Figueira and Nies~\cite{Figueira.Nies:13} and independently Kawamura and Miyabe (2013) have adapted the results on computable randomness in Sections~\ref{s:CRand_intro} and~\ref{s:CRandDiff}  to the subrecursive case, and in particular to  polynomial time randomness. There is an extensive theory of polynomial time computable functions on the unit interval; see for instance near the end of \cite{ Weihrauch:00}. 

%
%

\n \emph{Values of the derivative.}
If $f$ is a (Markov) computable function of bounded variation, then $f'(z) $ exists for all \ML{} random reals $z$ by Theorems~\ref{thm:MLbV} and~\ref{thm:extensionIQcomputable}. A.\ Pauly (2011) has asked what can be said about effectivity properties of the derivative as a function on the \ML{} random reals. Layerwise computability  in the sense of Hoyrup and Rojas~\cite{Hoyrup.Rojas:09} might be relevant here. Demuth has shown in \cite[p.\ 584]{Demuth:75} that if $f$ is Markov computable and $z$ is $\DII$ and satisfies a certain randomness property stronger than \ML's,  then  $f'(z)$ is    $\DII$ uniformly in  an index for~$z$ as a $\DII$ real; also see~\cite[Section~4]{Kucera.Nies:12}.
Similar questions can be asked about   other randomness notions and the corresponding classes of functions.

\n \emph{Extending  the results to higher dimensions.}
Several researchers have considered  extensions  of the results in this paper to higher dimensions. Already   Pathak~\cite{Pathak:09} showed that a  weak form of the Lebesgue   differentiation theorem  holds for \ML{} random points in the $n$-cube $[0,1]^n$. The above-mentioned    work  of Rute~\cite{Rute:12}, and Pathak, Simpson, and Rojas  strengthens this to Schnorr random points in the $n$-cube. On the other hand, functions of bounded variation can be defined in higher dimension \cite[p.\ 378]{Bogachev.vol1:07}, and    one might try to characterize \ML{} randomness in higher dimensions via their differentiability.
For weak $2$-randomness, recent   work of Galicki, Nies and Turetsky yields  the analog of Theorem~\ref{thm:w2rChar} in higher  dimensions. 

Rademacher's Theorem implies that a   Lipschitz function  on $[0,1]^n$ is almost everywhere differentiable.   Recall from the discussion in Section~\ref{ss:Pathak}  that computable randomness can be characterized via differentiability of computable Lipschitz functions defined on $[0,1]$.  Call a point $x =(x_1, \ldots, x_n)$ in  $[0,1]^n$  computably random if  no computable martingale succeeds on the   binary expansions of  $x_1, \ldots, x_n$ joined  in the canonical way (alternating between the sequences).   An obvious question is  whether  also higher dimensions, computable  randomness   is   equivalent to differentiability at~$x$ of all computable Lipschitz functions.   Galicki, Nies and Turetsky have announced an affirmative answer for one implication, that randomness implies differentiability. The converse implication remains open.

  {\bf Acknowledgment.} We would like to thank Santiago Figueira, Jason Rute and Stijn Vermeeren for the careful reading of the paper, and Anton\'in \Kuc\ for making Demuth's work accessible to us.

%

\begin{thebibliography}{10}

\bibitem{Bienvenu.Hoelzl.ea:12}
L.~Bienvenu, R.~Hoelzl, J.~Miller, and A.~Nies.
\newblock The {D}enjoy alternative for computable functions.
\newblock In {\em STACS}, pages 543 -- 554, 2012.

\bibitem{Bogachev.vol1:07}
V.~I. Bogachev.
\newblock {\em Measure theory. {V}ol. {I}, {II}}.
\newblock Springer-Verlag, Berlin, 2007.

\bibitem{Brattka.Hertling.ea:08}
V.~Brattka, P.~Hertling, and K.~Weihrauch.
\newblock A tutorial on computable analysis.
\newblock In S.~Barry Cooper, Benedikt L\"owe, and Andrea Sorbi, editors, {\em
  New Computational Paradigms: Changing Conceptions of What is Computable},
  pages 425--491. Springer, New York, 2008.

\bibitem{Bruckner:78}
A.~M. Bruckner.
\newblock {\em Differentiation of real functions}, volume 659 of {\em Lecture
  Notes in Mathematics}.
\newblock Springer, Berlin, 1978.

\bibitem{Carothers:00}
N.L. Carothers.
\newblock {\em {Real analysis}}.
\newblock Cambridge University Press, 2000.

\bibitem{Day:09}
A.~Day.
\newblock Process and truth-table characterizations of randomness.
\newblock Unpublished, 20xx.

\bibitem{Demuth:75}
O.~Demuth.
\newblock The differentiability of constructive functions of weakly bounded
  variation on pseudo numbers.
\newblock {\em Comment. Math. Univ. Carolin.}, 16(3):583--599, 1975.
\newblock Russian.

\bibitem{Downey.Griffiths.ea:04}
R.~Downey, E.~Griffiths, and G.~Laforte.
\newblock On {S}chnorr and computable randomness, martingales, and machines.
\newblock {\em MLQ Math. Log. Q.}, 50(6):613--627, 2004.

\bibitem{Downey.Hirschfeldt:book}
R.~Downey and D.~Hirschfeldt.
\newblock {\em Algorithmic randomness and complexity}.
\newblock Springer-Verlag, Berlin, 2010.
\newblock 855 pages.

\bibitem{Downey.Hirschfeldt.ea:05}
R.~Downey, D.~Hirschfeldt, J.~Miller, and A.~Nies.
\newblock Relativizing {C}haitin's halting probability.
\newblock {\em J. Math. Log.}, 5(2):167--192, 2005.

\bibitem{LogicBlog:13}
A.~Nies (editor).
\newblock Logic {B}log.
\newblock Available at \url{http://dl.dropbox.com/u/370127/Blog/Blog2013.pdf},
  2013.

\bibitem{Figueira.Nies:13}
S.~Figueira and A.~Nies.
\newblock Feasible analysis, randomness, and base invariance.
\newblock To appear in {T}heory of {C}omputing {S}ystems.

\bibitem{Fowler.Preiss:09}
T.~Fowler and D.~Preiss.
\newblock A simple proof of {Z}ahorski's description of non-differentiability
  sets of {L}ipschitz functions.
\newblock {\em Real Anal. Exchange}, 34(1):127--138, 2009.

\bibitem{Freer.Kjos.ea:nd}
C.~Freer, B.~Kjos-Hanssen, A.~Nies, and F.~Stephan.
\newblock Effective aspects of {L}ipschitz functions.
\newblock {S}ubmitted.

\bibitem{Hoyrup.Rojas:09}
M.~Hoyrup and C.~Rojas.
\newblock Computability of probability measures and {M}artin-{L}\"of randomness
  over metric spaces.
\newblock {\em Inform. and Comput.}, 207(7):830--847, 2009.

\bibitem{Kucera.Nies:12}
Anton\'{\i}n Ku\v{c}era and Andr{\'e} Nies.
\newblock Demuth's path to randomness.
\newblock In {\em Proceedings of the 2012 international conference on
  Theoretical Computer Science: computation, physics and beyond}, WTCS'12,
  pages 159--173, Berlin, Heidelberg, 2012. Springer-Verlag.

\bibitem{Lebesgue:1909}
H.~Lebesgue.
\newblock Sur les int\'egrales singuli\`eres.
\newblock {\em Ann. Fac. Sci. Toulouse Sci. Math. Sci. Phys. (3)}, 1:25--117,
  1909.

\bibitem{Miyabe:12}
K.~Miyabe.
\newblock Characterization of {K}urtz randomness by a differentiation theorem.
\newblock {S}ubmitted, 2011.

\bibitem{Montalban:11}
Antonio Montalb{\'a}n.
\newblock Open questions in reverse mathematics.
\newblock {\em Bull. Symbolic Logic}, 17(3):431--454, 2011.

\bibitem{Nies:book}
A.~Nies.
\newblock {\em Computability and randomness}, volume~51 of {\em Oxford Logic
  Guides}.
\newblock Oxford University Press, Oxford, 2009.

\bibitem{Nies:ICM}
A.\ Nies.
\newblock Interactions of computability and randomness.
\newblock In {\em Proceedings of the International Congress of Mathematicians},
  pages 30--57. World Scientific, 2010.

\bibitem{Pathak:09}
N.~Pathak.
\newblock A computational aspect of the {L}ebesgue differentiation theorem.
\newblock {\em J. Log. Anal.}, 1:Paper 9, 15, 2009.

\bibitem{Pour-El.Richards:89}
M.~Pour-El and J.~Richards.
\newblock {\em Computability in analysis and physics}.
\newblock Perspectives in Mathematical Logic. Springer-Verlag, Berlin, 1989.

\bibitem{Rudin:74}
W.~Rudin.
\newblock {\em Real and complex analysis}.
\newblock McGraw-Hill Book Co., New York, second edition, 1974.

\bibitem{Rudin:87}
W.~Rudin.
\newblock {\em Real and complex analysis}.
\newblock McGraw-Hill Book Co., New York, third edition, 1987.

\bibitem{Rute:12}
J.~Rute.
\newblock Algorithmic randomness, martingales, and differentiability {I}.
\newblock {I}n preparation, 2012.

\bibitem{Schnorr:75}
C.P. Schnorr.
\newblock {\em Zuf\"alligkeit und {W}ahrscheinlichkeit. {E}ine algorithmische
  {B}egr\"undung der {W}ahrscheinlichkeitstheorie}.
\newblock Springer-Verlag, Berlin, 1971.
\newblock Lecture Notes in Mathematics, Vol. 218.

\bibitem{Weihrauch:00}
K.~Weihrauch.
\newblock {\em Computable Analysis}.
\newblock Springer, Berlin, 2000.

\bibitem{Zahorski:46}
Z.~Zahorski.
\newblock Sur l'ensemble des points de non-d\'erivabilit\'e d'une fonction
  continue.
\newblock {\em Bull. Soc. Math. France}, 74:147--178, 1946.

\end{thebibliography}


\end{document}